\documentclass[12pts]{amsart}
\usepackage{amsthm,amsmath,stmaryrd,bbm,amssymb, mathrsfs, amsbsy, dsfont}
\usepackage[dvipsnames]{xcolor}
\usepackage[english]{babel}
\usepackage[utf8]{inputenc}

\usepackage{graphicx}
\usepackage{verbatim}
\usepackage{enumitem}
\usepackage{etoolbox}
\usepackage[colorlinks=true,linkcolor=blue, citecolor=blue]{hyperref}
\usepackage{todonotes}

\usepackage{tcolorbox}

\renewcommand{\H}{\mathbb H}
\newcommand{\R}{\mathbb R}

\renewcommand{\S}{\mathbb S}
\newcommand{\N}{\mathbb N}
\newcommand{\Z}{\mathbb Z}

\renewcommand{\AA}{\mathcal A}
\newcommand{\BB}{\mathcal B}
\newcommand{\CC}{\mathcal C}

\newcommand{\GG}{\mathcal G}
\newcommand{\HH}{\mathcal H}
\newcommand{\II}{\mathcal I}
\newcommand{\JJ}{\mathcal J}
\newcommand{\KK}{\mathcal K}
\newcommand{\LL}{\mathcal L}
\newcommand{\MM}{\mathcal M}
\newcommand{\NN}{\mathcal N}
\newcommand{\OO}{\mathcal O}
\newcommand{\QQ}{\mathcal Q}
\newcommand{\RR}{\mathcal R}
\renewcommand{\SS}{\mathcal S}
\newcommand{\TT}{\mathcal T}
\newcommand{\UU}{\mathcal U}
\newcommand{\VV}{\mathcal V}
\newcommand{\XX}{\mathcal X}

\newcommand{\MMM}{\mathscr M}

\theoremstyle{plain}
\newtheorem{theo}{Theorem}
\newtheorem{prop}[theo]{Proposition}
\newtheorem{lem}[theo]{Lemma}
\newtheorem{cor}[theo]{Corollary}
\theoremstyle{remark}
\newtheorem{rem}[theo]{Remark}
\theoremstyle{definition}

\newtheorem{hyp}{\textsc{Structural Assumptions}}

\numberwithin{equation}{section}
\numberwithin{theo}{section}

\def\le{\leqslant}
\def\ge{\geqslant}

\DeclareMathOperator{\id}{Id}

\def\eps{\varepsilon}
\renewcommand\d{\textnormal{d}}

\newcommand{\dom}{\mathscr{D}}

\DeclareMathOperator{\nul}{\mathrm{Ker}}
\DeclareMathOperator{\re}{\mathrm{Re}}

\newcommand{\ini}{\textnormal{in}}
\newcommand{\err}{\textnormal{err}}
\renewcommand{\wp}{\textnormal{wp}}
\newcommand{\ip}{\textnormal{ip}}
\newcommand{\ns}{\textnormal{NS}}

\newcommand{\hyd}{\textnormal{hydro}}

\newcommand{\kin}{\textnormal{kin}}

\newcommand{\Wave}{\textnormal{wave}}
\newcommand{\disp}{\textnormal{disp}}
\newcommand{\Inc}{\textnormal{inc}}
\newcommand{\Bou}{\textnormal{Bou}}

\def\la{\langle}
\def\ra{\rangle}
\makeatletter
\newsavebox{\@brx}
\newcommand{\lla}[1][]{\savebox{\@brx}{\(\m@th{#1\langle}\)}%
	\mathopen{\copy\@brx\kern-0.5\wd\@brx\usebox{\@brx}}}
\newcommand{\rra}[1][]{\savebox{\@brx}{\(\m@th{#1\rangle}\)}%
	\mathclose{\copy\@brx\kern-0.5\wd\@brx\usebox{\@brx}}}
\makeatother
\newcommand{\Nt}[1]{{\left\vert\kern-0.25ex\left\vert\kern-0.25ex\left\vert #1 
		\right\vert\kern-0.25ex\right\vert\kern-0.25ex\right\vert}}

\newcommand{\Ss}{V}       
\newcommand{\Ssp}{V^{\bullet}}    
\newcommand{\Ssm}{V^{\circ}}    
\newcommand{\rSSs}[1]{\VV^{#1}}	 		
\newcommand{\rhSSs}[1]{\dot{\VV}^{#1}}    
\newcommand{\rSSsp}[1]{\VV^{{\bullet}, #1}}     
\newcommand{\rSSsm}[1]{\VV^{{\circ}, #1}}    
\newcommand{\rhSSsm}[1]{\dot{\VV}^{{\circ}, #1}}    

\newcommand{\rSSSh}[1]{{\boldsymbol{\HH}}^{#1}}     
\newcommand{\rSSSk}[1]{\boldsymbol{\KK}^{#1}}     

\newcommand{\sumsp}[1]{{\boldsymbol{\XX}}^{#1}}

\newcommand{\hypst}[1]{\textnormal{\textbf{(#1)}}}

\makeatother

\title[Incompressible Navier-Stokes-Fourier limit of kinetic equations]{Incompressible Navier-Stokes limit of non-bilinear kinetic equations and application to the BGK, nonlinear Fokker-Planck and Boltzmann-Fermi-Dirac equations}

\author{Pierre Gervais}
\email{pierre.gervais@math.univ-toulouse.fr}

\newcommand{\natv}{\widetilde{\nabla}_v}

\begin{document}
	
	\maketitle
	
	\tableofcontents
	
	\begin{abstract}
		We consider collisional kinetic equations whose collision operator is not necessarily bilinear and prove quantitative convergence to the Navier-Stokes-Fourier system in weighted Sobolev spaces, together with a description of the initial layers. The aim of this paper is to conciliate the conditional convergence result of Bardos-Golse-Levermore \cite{BGLI91} for abstract kinetic equations conserving macroscopic quantities and dissipating entropy with the spectral strategy initiated by Bardos and Ukai \cite{BU91} for the Boltzmann equation. This work extends the abstract approach of \cite{GL24} which was restricted to bilinear collisions (Boltzmann, Landau or quadratic approximation of other models) to non-bilinear models such as the Boltzmann-Fermi-Dirac  equation, the BGK equation and the nonlinear Fokker-Planck equation.
	\end{abstract}
	
	\section{Introduction}
		\subsection{Presentation of the problem}
		
		We consider a generic kinetic equation with no force and local collisions in diffusive scaling:
		\begin{equation}
			\label{eq:generic_kinetic_equation}
			\eps \partial_t F + v \cdot \nabla_x F = \frac{1}{\eps} \CC[F] \ .
		\end{equation}
		In their seminal work \cite{BGLI91}, Bardos, Golse and Levermore established formally that any kinetic equation conserving mass, momentum, energy, and dissipating entropy towards maxwellian distributions, behaved asymptotically as the solutions to various fluid equations depending on the scaling considered. Specifically, for a diffusive scaling as \eqref{eq:generic_kinetic_equation}, fluctuations of size $\eps$ around a spatially homogeneous equilibrium are governed by the incompressible Navier-Stokes equations.
		
		\medskip
		Consider a spatially homogeneous equilibrium $\MM = \MM(v)$ (in the sense that $\CC[\MM] = 0$) and fluctuations of order $\eps$, that is $F(t, x, v) = \MM(v) + \eps f^\eps\left(  t , x , v \right)$, so that the profile $f^\eps$ satisfies the perturbation equation posed in $[0, \infty)_t \times \R^{2d}_{x,v}$
		\begin{equation}
			\label{eq:perturbed_scaled}
			\partial_t f^\eps = \frac{1}{\eps^2} \left( \LL - \eps v \cdot \nabla_x \right) f^\eps + \frac{1}{\eps} \QQ(f^\eps, f^\eps) +  \RR^\eps[f^\eps] \ , 
		\end{equation}
		where $\LL$, $\QQ$ and $\RR^\eps$ are given by the Taylor expansion of the collision operator:
		$$\CC[ \MM + \eps \varphi ] = \eps \LL \varphi + \eps^2 \QQ(\varphi, \varphi) + \eps^3 \RR^\eps[ \eps \varphi] \ .$$
		We recall the formal convergence result of Bardos, Golse and Levermore \cite[Theorem III]{BGLI91}.
		\begin{theo}
			\label{thm:BGL}
			Assume $\MM(v) = (2\pi)^{-\frac{d}{2}} e^{-\frac{|v|^2}{2}} $ and that $\CC$ formally conserves mass, momentum and energy:
			$$\forall \varphi \ , \quad \int_{\R^d} \CC[ \varphi ] (1, v_1, \dots, v_d, |v|^2) \d v = (0, \dots, 0) \ ,$$
			dissipates entropy:
			$$\forall \varphi \ , \quad \int_{ \R^d } \CC[\varphi] \log \varphi \ \d v \le 0 \ ,$$
			and vanishes on Maxwellians:
			$$\forall \varphi \ , \quad \left( \CC[\varphi]  = 0 \quad \Leftrightarrow \quad \varphi(v) = \frac{R_\varphi}{ (2\pi)^{\frac{d}{2}} } \exp \left(- \frac{|v-U_\varphi|^2}{2 T_\varphi} \right) \right) \ ,$$
			where $R_\varphi, U_\varphi, T_\varphi$ are the local mass density, velocity and temperature of $\varphi$:
			$$R_\varphi = \int_{ \R^d } \varphi \d v \ , \quad R_\varphi U_\varphi = \int_{ \R^d } v \varphi \d v \ , \quad R_\varphi \left( |U_\varphi|^2 + d T_\varphi \right) = \int_{ \R^d } |v|^2 \varphi \d v \ .$$
			If the kinetic equation \eqref{eq:perturbed_scaled} admits a solution $f^\eps$ such that $( f^\eps , \RR^\eps[f^\eps] ) \xrightarrow[\eps \to 0]{} (f^0, 0)$ in some appropriate sense, then 
			$$(2\pi)^{\frac{d}{2}} e^{\frac{|v|^2}{2}} f^0(t, x, v) = \rho(t, x) + v \cdot u(t, x) + \theta(t, x) \frac{ |v|^2- d}{2}$$
			where $(\rho, u, \theta)$ solves the incompressible Navier-Stokes-Fourier system:
			\begin{equation*}
				\begin{cases}
					\partial_t u + u \cdot \nabla u = \alpha \Delta u + \nabla p \ ,  & \nabla \cdot u = 0 \ , \\
					\partial_t \rho + u \cdot \nabla \rho = \beta \Delta \rho \ ,  & \nabla (\rho + \theta ) = 0 \ ,
				\end{cases}
			\end{equation*}
			for some positive constants $\alpha, \beta > 0$ depending only on $\LL$.
		\end{theo}
		
		This result highlights the fact that the Navier-Stokes limit of kinetic equations conserving macroscopic quantites and dissipating entropy is a universal phenomenon, and reduces this issue to merely proving existence of kinetic solutions and proving convergence or compactness of enough moments of $f^\eps$, giving rise to two main lines of research. 
		
		The first one dealt with the Boltzmann equation, proving weak convergence of weak solutions far from equilibrium \cite{BGLII93, GSR04, GSR09, LM10, A12}. The second one dealt with strong convergence of smooth solutions close to equilibrium \cite{BU91, GT20, G23, B15, BMM19, CC26} and were based on a spectral analysis of the linearized equation uniformly in $\eps$ \cite{EP75, YY16}. These results were later extended to other kinetic equations with the same conservation laws and dissipating entropy, such as the Landau equation \cite{R21, CRT22}, the BGK model \cite{SR03}, nonlinear Fokker-Planck equation with constant temperature \cite{CJ24}, or Boltzmann-Fermi-Dirac model \cite{JXZ22}.
		
		Proving the stability of equilibria for such kinetic equations follow essentially the same strategy: controlling the nonlinearity using the energy norm and the norm dissipated by the linear part. With this observation, the author and B. Lods developped in \cite{GL24} a unified framework to derive the Navier-Stokes-Fourier system from any bilinear kinetic equation with the same conservation laws as the Boltzmann equation and satisfying a quantified version of the entropy dissipation principle ($\RR^\eps=0$ in \eqref{eq:perturbed_scaled}, \textit{e.g.} Boltzmann, Landau, or quadratic approximations of other equations) and revisited the spectral study \cite{EP75, YY16} with modern and quantitative arguments. Based on this work, the regularity assumption on the initial data was later relaxed in \cite{CGT26} for the Landau and Boltzmann equations.
		
		The current work aims at extending \cite{GL24} to any non-bilinear ($\RR^\eps \neq 0$) kinetic equation satisfying appropriate continuity estimates, thus allowing to consider the nonlinear Fokker-Planck model, BGK model, or Boltzmann-Fermi-Dirac. Note that, although such a connection is already known for the last two models, the current work aims at being a strong and quantified version of the Bardos-Golse-Levermore Theorem \ref{thm:BGL} together with a complete description of the initial layers. We also sensibly simplify the core of the proof by introducing a functional space suited to the multiscale structure of \eqref{eq:perturbed_scaled}.

		\subsection{Notations and abstract framework}
		
		For a measurable function $m = m(v) \ge 0$, we define the weighted Lebesgue spaces $L^p_v(m)$ through its norm
		$$\| f \|_{ L^p_v(m) } := 
		\begin{cases}
			\displaystyle \left( \int_{ \R^d } | f(v) |^p \d v \right)^{1/p}  \ , & 1 \le p < \infty \ , \\
			\\
			\displaystyle\sup_{ v \in \R^d } | f(v) | \ ,&  p = \infty \ ,
		\end{cases}$$
		and the weighted Sobolev space $H^n_v(m)$
		$$\| f \|_{ H^n(m)_v }^2 := \sum_{ k=0 }^n \| \nabla_v^k f \|_{ L^2(m) } \ .$$
		In this work, we will not consider weights depending on space, so we use the usual Lebesgue space $L^p_x$, and we will denote $\H^s$ (resp. $\dot{\H}^s$) for the inhomogeneous (resp. homogeneous) Sobolev spaces, which we reall to be defined through the norms
		$$\| f \|_{ \H^s }^2 := \int_{ \R^d } | \widehat{f}(\xi) |^2 \la \xi \ra^{2s} \d \xi \ , \qquad \| f \|_{ \H^s }^2 := \int_{ \R^d } | \widehat{f}(\xi) |^2 | \xi |^{2s} \d \xi \ ,$$
		where $\widehat{f}$ denotes the Fourier transform with respect to the spatial variable $x \in \R^d$.

		\medskip
		
		Consider some measurable weight function $\mu : \R^d \to (0, \infty)$, we will consider nested velocity spaces
		$$\Ssp \subset \Ss = L^2( \mu^{-1} ) \subset \Ssm := \left( \Ssp \right)'$$
		where the identification is performed through the inner product of $\Ss$, endowed with the norm
		$$\| f \|_{ \Ssm } = \sup_{ \| \varphi \|_{\Ssp} \le 1} \la f , \varphi \ra_\Ss \, .$$
		We will consider the following position-velocity Sobolev spaces:
		$$\rSSs{s} = \H^s_x \left( \Ss_v \right) \, , \quad \rSSsp{s} = \H^s_x \left( \Ssp_v \right) \, , \quad \rSSsm{s} = \H^s_x \left( \Ssm_v \right) \, , \quad \rhSSsm{s} = \dot{\H}^s_x \left( \Ssm_v \right)$$
		endowed with the norms
		$$\| f \|^2_{ \H^s_x ( X_v ) } := \int_{ \R^d } \| \widehat{f}(\xi) \|^2_{X_v} \la \xi \ra^{2s} \d \xi \ , \quad X = \Ss, \Ssp, \Ssm \ ,$$
		$$\| f \|^2_{ \dot{\H}^s_x ( \Ssm_v ) } := \int_{ \R^d } \| \widehat{f}(\xi) \|^2_{\Ssm_v} | \xi |^{2s} \d \xi \ .$$
		\medskip
		
		We now present our assumptions regarding $\LL$ and $\mu$.
		
		\begin{hyp}\label{AsL} The linear operator $\LL\: : \dom(\LL) \subset \Ss \to \Ss$ satisfies the following assumptions.
			\begin{enumerate}[label=\hypst{L\arabic*}]
				\item \label{L1} The operator $\LL$ is self-adjoint in $\Ss$ and commutes with orthogonal matrices:
				\begin{gather*}
					\langle \LL f, g \rangle_{\Ss} = \langle f, \LL g \rangle_{\Ss} = \langle \LL ({\Theta}f), \Theta g \rangle_\Ss,
				\end{gather*}
				for any $f, g \in \dom(\LL)$ and orthogonal matrix $\Theta \in \MMM_{d \times d}(\R)$, where $[\Theta f](v):=f(\Theta v)$.
				\item \label{L2} The weight function $\mu$ is positive, normalized, radial, and such that:
				\begin{gather*}
					\mu=\mu(|v|) > 0, \qquad \int_{\R^d} \mu(v) \d v = 1,\\
					E := \int_{\R^d} |v|^2 \mu(v) \d v < \infty, \qquad K := \frac{1}{E^2} \int_{\R^d} | v |^4 \mu(v) \d v < \infty.
				\end{gather*}
				\item \label{L3} The null-space of $\LL$ is given by
				\begin{equation*} 
					\nul\left( \LL \right) =\mathrm{Span}\left\{ \mu, v_1 \mu,  v_2 \mu, \dots, v_d \mu, |v|^2 \mu \right\}\end{equation*}
				and its domain $\dom(\LL)$ satisfies for some Hilbert space $\Ssp$
				\begin{equation*}
					\dom(\LL) \subset \Ssp \subset \Ss, \qquad \| \cdot \|_{\Ss} \le \| \cdot \|_{\Ssp},
				\end{equation*}
				and such that there holds for some $\lambda_{\LL} > 0$
				\begin{equation*}
					\la \LL f, f \ra_{\Ss} \le - \lambda_\LL \| f \|^2_{\Ssp} \qquad \text{ for any } f \in \dom(\LL) \cap \nul\left( \LL \right)^\perp\,.
				\end{equation*}
				\item \label{L4} The operator $\LL$ can be decomposed as 
				$$\LL = \BB+ \AA \qquad \text{with} \qquad \AA  : \Ss \to \Ss \ \text{bounded} ,$$
				where the splitting is compatible with a hierarchy of Hilbert spaces $(\Ss_j)_{j=0}^{2}$ such that
				\begin{enumerate}
					\item \label{assumption_hierarchy} the spaces $\Ss_j$ continuously and densely embed into one another:
					$$\Ss_{2} \hookrightarrow \Ss_1 \hookrightarrow \Ss_0 = \Ss,$$
					\item \label{assumption_multi-v} the multiplication by $v$ is bounded from $\Ss_{j+1}$ to $\Ss_{j}$, i.e.
					\begin{equation*}
						\|v f\|_{\Ss_{j}} \lesssim \|f\|_{\Ss_{j+1}} \qquad f \in \Ss_{j+1}, \quad j=0,1,
					\end{equation*}
					\item \label{assumption_bounded_A} the operator $\AA : \Ss_j \rightarrow \Ss_{j+1}$ is bounded for $j=0,1$.
					\item \label{assumption_dissipative} there exists $\lambda_\BB \geq \lambda_{\LL}$ such that, for $j=0,1,2$, 
					$$\forall \xi \in \R^d \ , ~ \re z > - \lambda_\BB \ , \quad \left\| z - \BB + i v \cdot \xi \right\|_{\Ss_j \to \Ss_j} \lesssim | \re z + \lambda_{\BB}|^{-1} \  .$$
				\end{enumerate}
			\end{enumerate}
		\end{hyp}
		
		In the sequel, we will denote by $\Pi$ the orthogonal projection from $\Ss$ onto $\ker(\LL)$, namely
		\begin{equation}
			\label{eq:def_Pi}
			\Pi f(v) = \left( \rho_f + u_f \cdot v + \frac{ \theta_f }{E (K-1)} \left( |v|^2-E \right) \right) \mu(v)
		\end{equation}
		where the macroscopic fields are defined as
		\begin{subequations}
			\label{eq:def_macro_fields}
			\begin{gather}
				\rho_f = \int_{ \R^d } f(v) \d v = \la f , \mu \ra_\Ss \, , \\
				u_f = \frac{d}{E} \int_{ \R^d } v f(v) \d v = \frac{d}{E} \la f , v \mu \ra_\Ss \, , \\
				\theta_f = \frac{1}{ E }\int_{ \R^d } f(v) \left( |v|^2-E \right) \d v = \frac{1}{E} \la f , \mu \ra_\Ss \, ,
			\end{gather}
		\end{subequations}
		and we will also denote the microscopic part of $f$ as
		\begin{equation}
			\label{eq:def_Pi_perp}
			f^\perp =: f - \Pi f \, .
		\end{equation}
		
		We now present the assumptions on the bilinear part of the collision operator.
		
		\begin{hyp}\label{AsB} The bilinear operator $\QQ$ satisfies the following assumptions:
			\begin{enumerate}[label=\hypst{B\arabic*}]
				\item \label{Bortho} The bilinear operator is $\Ss$-orthogonal to the null-space of $\LL$:
				$$\la \QQ(f, g), \mu \ra_{\Ss} = \la \QQ(f, g), v \mu \ra_{\Ss} = \la \QQ(f, g), |v|^2 \mu \ra_{\Ss} = 0,$$
				or, equivalently, in terms of integrals:
				$$\int_{\R^d} \QQ(f, g)(v) \d v = \int_{\R^d} v \QQ(f, g)(v) \d v = \int_{\R^d} |v|^2 \QQ(f, g)(v) \d v = 0.$$
				\item \label{Bisotrop} The bilinear operator commutes with orthogonal matrices:
				$$\la \QQ(f, g), h \ra_{\Ss} = \la \QQ\left( \Theta f, \Theta g \right), \Theta h \ra_{\Ss},$$
				for any orthogonal matrix $\Theta \in \MMM_{d \times d}(\R).$
				\item \label{Bbound} The bilinear operator satisfies the following dual estimate
				$$\| \QQ(f, g) \|_{\Ssm} \lesssim \| f \|_{\Ss} \| g \|_{\Ssp} + \| f \|_{\Ssp} \| g \|_{\Ss} \, .$$
			\end{enumerate}
		\end{hyp}
		
		Compared to the work \cite{GL24}, we add an extra assumption concerning the nonlinear remainder.
		
		\begin{hyp}\label{AsN} The nonlinear remainer $\RR^\eps$ is such that for any regularity index $s > \frac d2$ and some $\boldsymbol{\alpha} \in \left( \frac12, 1\right)$ if $d=2$ or $\boldsymbol{\alpha}=1$ if $d=3$, the following holds for any $f, g \in \rSSs{s}$ satisfying $ \| (f, g) \|_{ \rSSsp{s} }\le \delta \ll 1$
			\begin{enumerate}[label=\hypst{N\arabic*}]
				\item \label{AsN_ortho} The orthogonality property for any $x \in \R^d$
				$$\la \RR^\eps[f](x) , \mu \ra_\Ss = \la \RR^\eps[f](x) , v \mu \ra_\Ss =\la \RR^\eps[f](x) , |v|^2 \mu \ra_\Ss =0 \, .$$
				\item \label{AsN_bound} The uniform bound
				$$\| \RR^\eps[ f ] \|_{ \rSSsm{s} \cap \rhSSsm{1-\boldsymbol{\alpha}} } \lesssim \left\| | \nabla_x |^{ \boldsymbol{\alpha} } f \right\|_{ \rSSsp{s-\boldsymbol{\alpha}} } \| f \|_{\rSSs{s}} $$
				
				\item \label{AsN_Lip} The Lipschitz estimate
				\begin{align*}
					\| \RR^\eps[ f ] - \RR^\eps[ g ] \|_{\rSSsm{s} \cap \rhSSsm{1-\boldsymbol{\alpha}}} \lesssim & \left\| (f, g) \right\|_{\rSSs{s}} \left\| | \nabla_x |^{\boldsymbol{\alpha}} (f - g) \right\|_{ \rSSsp{s - \boldsymbol{\alpha}} }  \\
					& + \left\| | \nabla_x |^{\boldsymbol{\alpha}} (f, g) \right\|_{ \rSSsp{s-\boldsymbol{\alpha}} } \left\| f - g \right\|_{ \rSSs{s} } \ .
				\end{align*}
			\end{enumerate}
		\end{hyp}
		
		\begin{rem}
			The role of $\boldsymbol{\alpha}$ is to control small frequencies. We assume in Theorem \ref{thm:main} that the initial data lies in $\dot{\H}^{1-\boldsymbol{\alpha}}_x$, which is is a more flexible assumption than it lies in $L^p_x$ as made in \cite{GT20}.
		\end{rem}
		
		When $\RR$ is a trilinear (or more) operator, a sufficient condition for it to satisfy Assumptions \ref{AsN} is given below. The proof is postponed to Proposition \ref{prop:AsN_mult_suff}.
		\begin{hyp}
			Assume the operator $\RR$ is $N$--linear with $N \ge 3$ and local in $x$, then it satisfies \ref{AsN_bound} and \ref{AsN_Lip} under the following sufficient condition:
			\begin{enumerate}[label=\hypst{N'}]
				\item \label{AsN_suff} For any $( f_n )_{n=1}^N \subset \Ssp$, there holds
				$$\| \RR[f_1, f_2, \dots, f_N] \|_{ \Ssm } \lesssim  \left\| \overline{f}_1 \right\|_{ \Ssp } \prod_{n=2}^{N} \left\| \overline{f}_n \right\|_{ \Ss } \, ,$$
				where the right hand side is a sum over all permutation $(\overline{f}_1, \dots, \overline{f}_N )$ of $(f_1, \dots, f_N)$.
			\end{enumerate}
		\end{hyp}
	
		\subsection{Statement of the main result}
		
			We now state the main result of this paper, and in the following subsections we will present the models to which the abstract hydrodynamic limit Theorem \ref{thm:main} applies.
					
			\medskip
			We will state our hydrodynamic limit result which deals with small ill-prepared initial data. To this end, we split the initial data $f_\ini$ as
			$$f_\ini = \Pi f_\ini + f_\ini^\perp = f_{\ini, \wp} + f_{\ini, \ip} + f^\perp_\ini$$
			where the \textit{well-prepared part} $f_{\ini, \wp}$ and \textit{ill-prepared part} $f_{\ini, \ip}$ of $\Pi f_\ini$ write
			$$\mu(v)^{-1} f_{\ini, \star}(x, v) = \rho_{\ini, \star}(x) + u_{\ini, \star}(x) \cdot v + \frac{1}{E(K-1)} \theta_{\ini, \star}(x) \left( | v |^2 - E \right) $$
			with, using the notation of eq. \eqref{eq:def_macro_fields}
			$$\rho_{\ini, \ip} = \frac{\rho_\ini  + \theta_\ini}{K} \, , \quad \theta_{\ini, \ip} = \left(1-\frac{1}{K} \right) (\rho_\ini+\theta_\ini) \ , \quad u_{\ini, \ip} = | \nabla_x |^{-1} ( \nabla_x \cdot u_\ini ) \ ,$$
			as well as, denoting Leray's projector on incompressible fields by $\mathbb{P}$
			$$\rho_{\ini, \wp} = \rho_\ini - \rho_{\ini, \ip} \ , \qquad \theta_{\ini, \wp} = - \rho_{\ini, \wp} \ ,  \qquad u_{\ini, \wp} = u_\ini - u_{\ini, \ip} = \mathbb{P} u_\ini\ . $$
			Since $\|(\rho_{\ini, \wp}, u_{\ini, \wp}, \theta_{\ini, \wp}) \|_{\H^s} \lesssim \| f_{\ini, \wp} \|_{ \rSSs{s}} \lesssim \| f_\ini \|_{ \rSSs{s}} \ll 1$, we may also consider the unique global solution $(\rho_\ns, u_\ns, \theta_\ns) \in L^\infty_t \H^s_x$ spanned by $(\rho_{\ini, \wp}, u_{\ini, \wp}, \theta_{\ini, \wp})$ to the incompressible Navier-Stokes-Fourier system (see \cite{LR23}):
			\begin{equation}
				\label{eq:NSF}
				\begin{cases}
					\partial_t u_\ns + \vartheta_ \Inc u_\ns \cdot \nabla u_\ns = \kappa_\Inc \Delta u_\ns - \nabla p_\ns \, , & \nabla \cdot u_\ns = 0 \ , \\
					\partial_t \theta_\ns + \vartheta_\Bou u_\ns \cdot \nabla \theta_\ns = \kappa_\Bou \Delta \theta_\ns \, , & \rho_\ns + \theta_\ns = 0 \ ,
				\end{cases}
			\end{equation}
			which is such that $(\nabla \rho_\ns, \nabla u_\ns, \nabla \theta_\ns) \in L^2_t \H^s_x$. 
			
			\begin{theo}[Abstract hydrodynamic limit]
				\label{thm:main}
				Suppose that Assumptions \ref{AsL}, \ref{AsB}, \ref{AsN} hold and fix some regularity index $s > \frac{d}{2}$, there exists a small $\delta > 0$ such that the following holds for any initial data $f_\ini \in \rSSs{s}$ satisfying
				$$\| f_\ini \|_{ \rSSs{s} \cap \rhSSsm{1-\boldsymbol{\alpha}} } \le \delta \quad (d = 2) \ , \qquad
					\| f_\ini \|_{ \rSSs{s} } \le \delta \quad (d \ge 3) \ ,$$
				where $\boldsymbol{\alpha}$ is that of Assumption \ref{AsN_bound}--\ref{AsN_Lip}.
				
				\medskip
				\noindent
				\textbf{(1) Existence and uniqueness of a kinetic solution} The equation \eqref{eq:perturbed_scaled} admits a global solution $f^\eps$ which is unique among those such that
				$$\sup_{t \ge 0} \| f^\eps(t) \|_{ \rSSs{s} } \le 2 \delta \quad \text{and} \quad |\nabla_x|^{ \boldsymbol{\alpha} } f^\eps \in L^2_{loc} \left( [0, \infty) ; \rSSsp{s - \boldsymbol{\alpha} } \right) \, .$$
				
				\medskip
				\noindent
				\textbf{(2) Hydrodynamic limit and initial layers} The kinetic solution $f^\eps$ writes
				$$f^\eps = f_\ns + f^\eps_\kin + f^\eps_\disp + f^\eps_\err$$
				where each term of the decomposition above is as follows:
				\begin{itemize}
					\item The limit $\mu(v)^{-1} f_\ns(t,x,v)$ writes
					\begin{equation}
						\label{eq:macro_NSF}
						\rho_\ns(t,x) + u_\ns(t,x) \cdot v + \frac{1}{E(K-1)} \theta_\ns(t,x) \left( | v |^2 - E \right) \, .
					\end{equation}
					
					\item The initial layer is spanned by the microscopic part of the initial data:
					\begin{equation*}
						\| f^\eps_\kin(t) \|_{ \rSSs{s} } \lesssim e^{-\lambda t / \eps^2} \| f_\ini^\perp \|_{ \rSSs{s} } \, .
					\end{equation*}
					
					\item The oscillating part $f^\eps_\disp$ is generated by the ill-prepared part $f_{\ini, \ip}$ of the initial data:
					$$\| f^\eps_\disp(t) \|_{ L^\infty_t \rSSs{s} } + \| | \nabla_x |^{\boldsymbol{\alpha}} f^\eps_\disp \|_{ L^2_t \rSSs{s -\boldsymbol{\alpha} }} \lesssim \| f_{\ini, \ip} \|_{ \rSSs{s} \cap \rhSSsm{\boldsymbol{\alpha}-1} } \ , $$
					and vanishes in an averaged sense:
					$$\forall p > \frac{2}{d-1} \, , \quad \lim\limits_{\eps \to 0} \| f^\eps_\disp \|_{ L^p_t L^\infty_x \Ss_v } = 0$$
					as well as uniformly away from $t=0$:
					$$\forall \delta > 0 \, , \quad \lim\limits_{\eps \to 0} \left( \sup_{ t \ge \delta } \| f^\eps(t) \|_{ L^\infty_x \Ss_v } \right) = 0 \, .$$
					
					\item The error term $f^\eps_\err$ vanishes uniformly in time:
					$$\lim\limits_{\eps \to 0}  \| f^\eps_\err \|_{ L^\infty_t \rSSs{s} } = 0 \, .$$
				\end{itemize}
				If the macroscopic part of the initial data is well-prepared (i.e. $f_{\ini, \ip} = 0$ and thus~$f^\eps_\disp = 0$) then one has the quantitative estimate for $\delta \in (0, 1]$
				$$\| f^\eps_\err \|_{ L^\infty_t \rSSs{s} } \lesssim \eps^\delta \|f_\ini \|_{ \rSSs{s+\delta} } \, ,$$
				and in the ill-prepared case (i.e. $f_{\ini, \ip} \neq 0$) for $\delta \in \left[0, \frac12\right]$ and $\eta>0$
				$$\| f^\eps_\err(t) \|_{ L^\infty_t \rSSs{s} } \lesssim \eps^\delta \left( \| f_\ini \|_{ \rSSs{s+\delta} } + \| f_{\ini, \ip} \|_{ \dot{\mathbb{B}}^{ s + \frac{d+1}{2} + \eta}_{1,1} \cap \H^s } \right) \, ,$$
				$$\| f^\eps_\disp(t) \|_{ L^\infty_x \Ss_v } \lesssim  1 \land \left( \frac{\eps}{t} \right)^{ \frac{d-1}{2} } \| f_{\ini, \ip} \|_{ \dot{\mathbb{B}}^{ \frac{d+1}{2} }_{1,1} \cap \H^s } \, ,$$
				where $\mathbb{B}$ denotes the usual Besov space (see for instance \cite[Section 2]{BCD}).
			\end{theo}
			
			\begin{rem}
				A practical criterion for the initial data to be in $\rhSSsm{-\sigma}$ is that it lies in $L^p_x(\Ss_v)$ for some $1 < p < 2$:
				$$L^p_x ( \Ss_v ) \subset \rhSSs{-\sigma} \subset \rhSSsm{-\sigma} \quad \text{with} \quad \sigma = d \left(\frac1p - \frac12 \right) \ ,$$
				indeed, the $L^2$ structure of $\rhSSs{-\sigma}$ implies
				$$\| f \|_{ \rhSSs{-\sigma} }^2 = \int_{\R^{2d}} | \widehat{f}(\xi, v) |^2 \mu(v)^{-1} | \xi |^{-2\sigma} \d v \d \xi = \| f \|_{ L^2_v \dot{\H}^{-\sigma}_x(\mu^{-1}) }^2 $$
				thus, the embedding $L^p \subset \dot{\H}^{-\sigma}$ followed by Minkowski's inequality yield
				$$\| f \|_{ \rhSSs{-\sigma} } \lesssim \| f \|_{ L^2_v L^p_x(\mu^{-1}) } \lesssim  \| f \|_{ L^p_x (\Ss_v) } \ .$$
			\end{rem}
			
			\begin{rem}
				The strategy could be adapated to consider large initial data $f_\ini$. The kinetic solution $f^\eps$ would then be constructed for $\eps \le \eps_0(f_\ini)$ on any time interval $[0, T_{\text{max}} - \delta]$ where $T_{\text{max}} > 0$ is the maximal time of smoothness of the Navier-Stokes-Solution $(\rho_\ns, u_\ns, \theta_\ns)$. In particular, $T_{\text{max}} = \infty$ for $d=2$, thus the kinetic solution would be global. This approach has been adopted in \cite{GT20, GL24, CGT26} and we chose not to in the present work to avoid making the functional framework more complicated.
			\end{rem}
	
			\begin{rem}
				By asking more structure on the homogeneous steady states of the equation \eqref{eq:generic_kinetic_equation}, one may compute the transport coefficients $\vartheta_\star$ in \eqref{eq:NSF} by adapting the computations \cite[(57)--(62)]{BGLI91}. In particular, when these steady states are Maxwellians (as is the case for the Fokker-Planck and BGK models), then $\vartheta_\star = 1$.
			\end{rem}
	
	In the following sections, we will present various models which fall within the framework of Theorem \ref{thm:main} as will be proved in Sections \ref{scn:NFP}--\ref{scn:BFD}.

	\subsubsection{The BGK equation}
	\label{scn:BGK_intro}
	
	The Boltzmann collision operator is particularly delicate to compute numerically, for this reason an alternative operator, called the \textit{BGK operator}, was proposed in \cite{BGK54}. It satisfies the same conservation and entropy dissipation laws as the original operator while being much simpler:
	\begin{equation}
		\label{eq:def_BGK}
		\CC_{\text{BGK}}[F] = -( F -\MM[F] ) \ ,
	\end{equation}
	with $\MM[F]$ being the thermodynamic equilibrium with the same macroscopic fields as $F$:
	\begin{equation}
		\label{eq:local_maxwellian}
		\MM[F] = \frac{R }{ (2\pi T)^{d/2} } \exp \left( - \frac{|v-U|^2}{2 T} \right)
	\end{equation}
	where $R$, $U$ and $T$ are the mass density, velocity and temperature:
	\begin{gather}
		\label{eq:macro_quantities}
		R = \int F \, \d v \, , \qquad U = \frac{1}{R} \int v F \, \d v \, , \qquad T = \frac{1}{d R} \int |v-U|^2 F \, \d v \, .
	\end{gather}
	We know that this model has global solutions in $L^1$ \cite{P89} from which the incompressible Navier-Stokes-Fourier system was derived \cite{SR03}. In our work, we quantify this convergence and describe the initial layers.

	A well known limitation of this model is that it yields a wrong Prandtl number when deriving the (compressible) Navier-Stokes equations, so an alternative model, called the \textit{Gaussian} or \textit{Ellipsoidal-Statistical-BGK operator}, was proposed in \cite{ALTPP00}: 
	\begin{equation}
		\label{eq:def_ES_BGK}
		\CC_{\text{ES-BGK}}[F] = -(F - \GG[F]) \ .
	\end{equation}
	This time, the distribution $\GG[F]$ is a anisotropic gaussian centered around $U$ given by
	$$\GG[F] = \frac{R}{ \sqrt{ \det (2\pi \TT) } } \exp \left( - \frac12 (v-U) \TT^{-1} (v-U) \right)$$
	where the matrix $\TT$ is defined for some $\nu \in \left( - \frac{1}{2} , 1 \right)$ as
	$$\TT = (1-\nu) T \ \text{id} + \frac{\nu}{R} \int_{ \R^d } (v-U)^{\otimes 2} F \ \d v \ .$$
	Since proving that the BGK model \eqref{eq:def_BGK} falls within our framework already requires tedious computations, we treat the case of the ES-BGK model. However, because of their similar structure, we claim that this model is contained in our framework as well.
		
	\subsubsection{The nonlinear Fokker-Planck equation}
	
	We consider the \textit{nonlinear Fokker-Planck equation}, introduced in \cite{V02}, where the collision operator is the Fokker-Planck operator associated to the local Maxwellian $\MM[F]$ with the same macroscopic fields as $F$:
	\begin{equation}
		\label{eq:def_NFP}
		\CC_{\text{NFP}}[F] = T \nabla_v \cdot \left( \MM[F] \nabla_v \left( \frac{F}{\MM[F] } \right) \right) = T \Delta_v F + \nabla_v \cdot \left(  v F - U F\right) \ .
	\end{equation}
	This equation is well understood in the case of a constant velocity and temperature $U=0, T=1$, that is to say in the linear case. When the temperature is constant, this model is also known as the \textit{flocking model} \cite{C16, MT11} whose diffusive limit yields the incompressible Navier-Stokes system \cite{CJ24}. The literature dealing with non-constant velocity and temperature is much more scarce, although we know this model admits global weak solutions \cite{CHY25} and its equilibria are exponentially stable \cite{LY21}. Up to our knowledge, this manuscript presents the first derivation of the incompressible Navier-Stokes-Fourier system for this model.

	As for the BGK model, a variant called the \textit{Ellipsoidal-Statistical-Fokker-Planck model} was introduced in \cite{MM16}:
	$$\CC_{\text{ES-NFP}}[F] = \nabla_v \cdot \left( \TT \GG[F] \nabla_v \left( \frac{F}{\GG[F] } \right) \right) = \nabla_v \cdot \left( \TT \nabla_v F + (v-U) F \right) \ .$$
	We also claim that our framework allows to deal with this model.
	
	\subsubsection{The Boltzmann-Fermi-Dirac equation}
	We consider the \textit{Boltzmann-Fermi-Dirac equation} for hard potential interactions with an angular cutoff assumption
	\begin{equation}
		\label{eq:def_quantum_Boltzmann}
		\CC[F] = \CC^+[F] - \CC^-[F] 
	\end{equation}
	where the gain and loss terms are defined as
	\begin{gather*}
		\CC^+[F] = \int_{ \R^d_{v_*} \times \S^{d-1}_\sigma }  B( v - v _*, \sigma) F' F'_* (1-F) (1-F_*)  \d \sigma \d v_* \, , \\
		\CC^-[F] =  \int_{ \R^d_{v_*} \times \S^{d-1}_\sigma }  B( v - v _*, \sigma)  F F_* (1-F') (1-F_*')  \d \sigma \d v_* \, .
	\end{gather*}
	We have used the usual shorthand notation
	$$F' = F(v') \, , \qquad F'_* = F(v'_*) \, , \qquad F_* = F(v_*) \, ,$$
	where the pre-collisional velocities $(v', v'_*)$ are given by the formulae
	$$v' = \frac{v+v_*}{2} + \sigma \frac{|v-v_*|}{2} \, , \qquad v'_* = \frac{v+v_*}{2} - \sigma \frac{|v-v_*|}{2} \, .$$
	The collision kernel $B$ can be assumed to be of the form
	$$B(v-v_*, \sigma) = | v - v_*|^\gamma \, b\left( \sigma \cdot \frac{v-v_*}{|v-v_*|} \right)$$
	for some $\gamma\in [0, 1]$ and some measurable non-negative function that satisfies a \textit{cutoff assumption}:
	$$\forall u \in \S^{d-1} \, , \quad 0 < \int_{ \S^{d-1} } b( u \cdot \sigma ) \,  \d \sigma < \infty \, .$$
	Global solutions are known to exist in $L^1 \cap L^\infty$ \cite{D94} and its hydrodynamic limit in weighted $L^2$ spaces was derived using compactness arguments \cite{JXZ22}. In our work, we quantify this convergence and describe the initial layers.
	
	\subsection{Idea of the proof}
	
	The spectral study \cite{GL24} highlights that a relevant space to capture the multiscale nature of the equation is $\sumsp{s}_\eps = \rSSSk{s}_\eps + \rSSSh{s}$, where the ‘‘kinetic regime'' space $\rSSSk{s}_\eps$ corresponds to the energy estimate of $\partial_t h = \eps^{-2} \LL h$, and the ‘‘hydrodynamic regime'' $\rSSSh{s}$ corresponds to the energy estimate of $\partial_t h = \Delta_x h$. We then consider the kinetic equation \eqref{eq:perturbed_scaled} in its Duhamel formulation, substract the \textit{a priori} Navier-Stokes-Fourier limit \eqref{eq:macro_NSF} and the initial layers (kinetic and acoustic). We are then left with an integral equation on the error where the linear and non-linear operators are bounded in $\sumsp{s}_\eps$. We prove existence of a solution by a contraction argument and then quantify its rate of convergence.
		
	\subsection{Outline of the paper}
	
	In Section \ref{scn:hydro_limit}, we prove the abstract hydrodynamic limit Theorem \ref{thm:main}. We first present the functional spaces, then in Section \ref{scn:reminders} we present reminders of the previous work \cite{GL24}, including continuity estimates for the (non-)linear operators, and reformulate them in terms of our functional spaces, finally we prove the abstract hydrodynamic limit Theorem \ref{thm:main} in Section \ref{scn:proof_hydro_limit}.
	
	In the next sections, we prove that our framework applies to several non-linear collisional models, namely the nonlinear Fokker-Planck model \eqref{eq:def_NFP} in Section \ref{scn:NFP}, the BGK model \eqref{eq:def_BGK} in Section \ref{scn:BGK}, and the Boltzmann-Fermi-Dirac model \eqref{eq:def_quantum_Boltzmann} in Section \ref{scn:BFD}.
	
	In the appendix, we present Taylor expansions of the macroscopic quantities \eqref{eq:macro_quantities} and of the local Maxwellian \eqref{eq:local_maxwellian}, and then prove some nonlinear estimates in Sobolev spaces.
	
	\section{The abstract hydrodynamic limit theorem}
		\label{scn:hydro_limit}
	
		We introduce the time-position-velocity spaces in which we will work:
		For ‘‘kinetic'' functions, we consider
		$$\rSSSk{s} = \rSSSk{s}_\eps = \left\{ f \in \CC_b( [0, \infty) ; \rSSs{s} ) \, : \, \Nt{f}_{ \rSSSk{s} } < \infty \right\}$$
		endowed with the norm
		$$\Nt{f}_{ \rSSSk{s} }^2 = \sup_{t \ge 0} \| f(t) \|_{\rSSs{s}}^2 + \frac{1}{\eps^2} \int_0^\infty \| f(t) \|_{ \rSSsp{s} }^2  \d t$$
		For ‘‘hydrodynamic'' functions, we consider
		$$\rSSSh{s} = \left\{ f \in \CC_b( [0, \infty) ; \rSSs{s} ) \, : \, \Nt{f}_{ \rSSSh{s} } < \infty \right\}$$
		endowed with the norm
		$$\Nt{f}_{ \rSSSh{s} }^2 = \sup_{t \ge 0} \| f(t) \|_{\rSSsp{s}}^2 + \int_0^\infty \left\| | \nabla_x |^{ \boldsymbol{\alpha} } f(t) \right\|_{ \rSSsp{s-\boldsymbol{\alpha}} }^2  \d t$$
		where $\boldsymbol{\alpha}$ is that of Assumption \ref{AsN}. Note that, under the smallness assumption $\| f_\ini \|_{ \rSSs{s} } \le \delta \ll 1$, the hydrodynamic limit $f_\ns$ lies in $\rSSSh{s}$ and
		\begin{equation}
			\label{eq:nsf_hydro_regime_est}
			\| f_\ns \|_{ \rSSSh{s} } \lesssim \| \Pi f_\ini \|_{ \rSSs{s} } + \| \Pi f_\ini \|_{ \dot{\H}_x^{ \boldsymbol{\alpha} - 1 } ( \Ssm_v )  } \lesssim \delta \, .
		\end{equation}
		Finally we define the sum of these two spaces
		$$\sumsp{s} = \sumsp{s}_\eps = \rSSSh{s} + \rSSSk{s}_\eps $$
		endowed with the norm
		$$\| f \|_{ \sumsp{s} } = \inf \left\{ \|f_1\|_{ \rSSSh{s} } + \| f_2 \|_{ \rSSSk{s} } \, : \, (f_1,f_2) \in \rSSSh{s} \times \rSSSk{s} , \, f = f_1 + f_2 \right\} \, .$$
		Note that we have the comparison
		$$\sup_{ t \ge 0 } \| f(t) \|_{ \rSSs{s} } \lesssim \Nt{ f }_{ \sumsp{s} }$$
		uniformly in $\eps$.
		
		\subsection{Reminders and extension of the abstract theory}
		\label{scn:reminders}
		
		We consider the projectors on the different modes of the linearized equation \eqref{eq:perturbed_scaled}, introduced in \cite[Section 2.3]{GL24} which converge as $\eps \to 0$:
		$$\id = P^\eps_{\hyd} + P_\kin^\eps \ , \qquad P_\hyd^\eps = P_\ns^\eps + P^\eps_\Wave \ .$$
		Note that, at $\eps=0$, these projectors provide the decomposition of the initial data into its well-prepared part and ill prepared parts defined before Theorem \ref{thm:main}:
		$$f_{\ini, \wp} = P^0_\ns f_\ini \ , \qquad f_{\ini, \ip} = P^0_\Wave f_\ini \ , \qquad f_\ini^\perp = P^0_\kin \ .$$
		We also consider the semigroup $U^\eps$ generated by $\eps^{-2} \left( \LL - \eps v \cdot \nabla_x \right)$ and its splitting induced by the previous projectors:
		$$U^\eps = U_\hyd^\eps + U^\eps_\kin \, , \qquad U^\eps_\hyd = U^\eps_\ns + U^\eps_\Wave \, , \qquad U^\eps_\star = U^\eps P_\star^\eps \ ,$$
		as well as the dispersive semigroup introduced in \cite[Definition 2.8]{GL24}, which is an approximation of $U^\eps_\Wave$:
		$$U^\eps_\disp = U^\eps_\disp P_\Wave^0 \ .$$
		Finally, we consider the Duhamel bilinear term $\Psi^\eps$ defined as
		$$\Psi^\eps[f, g](t) = \frac{1}{\eps} \int_0^t U^\eps(t-\tau) \QQ(f(\tau), g(\tau)) \d \tau$$
		which follows a splitting induced by that of $U^\eps$:
		$$\Psi^\eps = \Psi^\eps_\hyd + \Psi^\eps_\kin \, , \qquad \Psi^\eps_\hyd = \Psi^\eps_\ns + \Psi^\eps_\Wave \ .$$
		As stated in \cite[Proposition 2.5]{GL24}, a kinetic distribution $f \in L^\infty_t \rSSsp{s}$ such that $\nabla_x f \in L^2 \rSSsp{s}$ describes a solution of Navier-Stokes-Fourier in the sense of \eqref{eq:NSF}--\eqref{eq:macro_NSF} if and only if it satisfies the integral equation
		\begin{equation}
			\label{eq:integral_NSF}
			f(t) = U_\ns(t) f(0) + \Psi_\ns[f](t)
		\end{equation}
		where $U_\ns = \lim\limits_{\eps \to 0} U^\eps_\ns$ and $\Psi_\ns=\lim\limits_{\eps \to 0} \Psi^\eps_\ns$ are defined in \cite[Definition 2.4]{GL24}.
		
		\medskip
		We now recall the main estimates for $U^\eps_\star$ and $\Psi^\eps_\star$ established in \cite{GL24}.
		
		\begin{prop}[Linear and bilinear estimates]
			One has for $\star=\ns, \Wave, \disp$
			\begin{equation}
				\label{eq:linear_semigroup_est}
				\Nt{ U^\eps_\star g }_{ \rSSSh{s} } \lesssim \| g \|_{ \rSSs{s} } + \| g \|_{ \rhSSsm{1-\boldsymbol{\alpha}} } \, \qquad \Nt{ U^\eps_\kin g }_{ \rSSSk{s} } \lesssim \| g \|_{ \rSSs{s} } \, .
			\end{equation}
			Furthermore, when both arguments are in the hydrodynamic regime $\rSSSh{s}$
			\begin{equation}
				\label{eq:bilinear_from_hydro}
				\Nt{ \Psi^\eps_\hyd[ f , g ] }_{ \rSSSh{s} } + \Nt{ \Psi^\eps_\kin[ f , g ] }_{ \rSSSk{s} } \lesssim \Nt{ f }_{ \rSSSh{s} } \Nt{ f }_{ \rSSSh{s} } \ ,
			\end{equation}
			and when $f$ (for instance) is in the kinetic regime then
			\begin{equation}
				\label{eq:bilinear_from_kinetic}
				\Nt{ \Psi^\eps_\hyd[ f , g ] }_{ \rSSSh{s} } + \Nt{ \Psi^\eps_\kin[ f , g ] }_{ \rSSSk{s} } \lesssim \eps \, \Nt{ f }_{ \rSSSk{s} } \min\left\{ \Nt{ g }_{ \rSSSh{s} } , \Nt{ g }_{ \rSSSk{s} } \right\} \ .
			\end{equation}
		\end{prop}
		The linear estimate comes from \cite[Lemma 4.1]{GL24}. The bilinear estimate where both arguments are in the hydrodynamic regime corresponds to \cite[(5.3a) and (5.14)]{GL24}, and the one where at least one argument is in the kinetic regime \cite[(5.3b) and (5.13)]{GL24}. Note that our space $\rSSSh{s}$ corresponds to that of \cite[Definition 2.1--(3)]{GL24} with $\phi = 0$ and thus $w_{\phi, \eta} = 1$.
		
		\begin{prop}[Oscillation estimates]
			There holds for $g \in \rSSsm{s} \cap \rhSSsm{1-\boldsymbol{\alpha}}$ and any $f$ in the hydrodynamic regime
			\begin{equation}
				\label{eq:psi_oscillation_est}
				\Nt{ \Psi^\eps\left[ f, U^\eps_\disp g \right] }_{ \sumsp{s} } \lesssim \beta(g, \eps) \Nt{ f }_{ \rSSSh{s} }
			\end{equation}
			where the vanishing rate $\beta$ is such that
			$$\eps \| g \|_{ \rSSs{s} } \lesssim \beta(g, \eps) \lesssim \| g \|_{ \rSSs{s} } + \| g \|_{ \rhSSsm{1-\boldsymbol{\alpha}} } \, , \qquad \beta(g, \eps) \xrightarrow{\eps \to 0} 0 \, .$$
		\end{prop}
		This was stated in \cite[eqs. (5.9) and (5.19)]{GL24}.
		
		In terms of the sum space $\sumsp{s}$, these estimates rewrite as follows.
		
		\begin{cor}[Bilinear estimates in the sum space]
			The bilinear operator $\Psi^\eps$ is continuous in the space $\sumsp{s}$:
			\begin{equation}
				\label{eq:bilinear_psi_est}
				\Nt{ \Psi^\eps[ f , g ] }_{ \sumsp{s} } \lesssim \Nt{ f }_{ \sumsp{s} } \Nt{ f }_{ \sumsp{s} } \ ,
			\end{equation}
			and, more precisely, when $f$ (for instance) is in the kinetic regime
			\begin{equation}
				\label{eq:bilinear_psi_kin_est}
				\Nt{ \Psi^\eps[ f , g ] }_{ \sumsp{s} } \lesssim \eps \, \Nt{ f }_{ \rSSSk{s} } \Nt{ g }_{ \sumsp{s} } \, .
			\end{equation}
			Furthermore, for any $g \in \rSSsm{s} \cap \rhSSsm{1-\boldsymbol{\alpha}}$, one has
			\begin{equation}
				\label{eq:convergence_psi_oscillation}
				\Nt{ \Psi^\eps\left[ f , U^\eps_\disp g \right] }_{ \sumsp{s} } \lesssim \beta( g , \eps ) \Nt{ f }_{ \sumsp{s} } \, .
			\end{equation}
		\end{cor}
%
		We now introduce the Duhamel term associated with the fully nonlinear term of the equation:
		\begin{equation}
			\begin{split}
				\Lambda^\eps[  g ](t) := \frac{1}{\eps} \int_0^t U^\eps(t - \tau) \RR^\eps[ g(\tau) ] \d \tau \\
				= \Lambda^\eps_\hyd[g](t) + \Lambda^\eps_\kin[ g](t) \, .
			\end{split}
		\end{equation}
		
		\begin{prop}
			For any $g, h \in \sumsp{s}$ such that $\displaystyle \sup_{t \ge 0} \| (g, h)(t) \|_{ \rSSs{s} } \le \delta$, there holds
			\begin{gather}
				\label{eq:fully_nonlin_bound}
				\Nt{ \Lambda^\eps[ g] }_{ \sumsp{s} } \lesssim \Nt{ g }_{ \sumsp{s} }^2 \, , \\
				\label{eq:fully_nonlin_Lip}
				\Nt{ \Lambda^\eps[  g] - \Lambda^\eps[ h] }_{ \sumsp{s} } \lesssim \Nt{ (g, h) }_{ \sumsp{s} } \Nt{ g - h }_{ \sumsp{s} } \, .
			\end{gather}
		\end{prop}
		
		\begin{proof}
			Consider some decomposition $g = g_\kin + g_\kin \in \rSSSh{s} + \rSSSk{s}$. Combining \cite[eq. (4.2a)]{GL24} with \ref{AsN_bound}, one has
			\begin{align*}
				\Nt{ \Lambda^\eps_\hyd[g] }^2_{ \rSSSh{s} } \lesssim & \int_0^\infty \left\| | \nabla_x |^{ \boldsymbol{\alpha} } g(t) \right\|_{ \rSSsp{s-\boldsymbol{\alpha}} }^2 \| g(t) \|_{\rSSs{s}}^2 \d t \\
				\lesssim & \| g \|_{ L^\infty_t ( \rSSs{s} ) }^2 \int_0^\infty \left\| | \nabla_x |^{ \boldsymbol{\alpha} } g_\hyd(t) \right\|_{ \rSSsp{s-\boldsymbol{\alpha}} }^2  \d t \\
				& +  \| g \|_{ L^\infty_t ( \rSSs{s} ) }^2 \int_0^\infty \left\| g_\kin(t) \right\|_{ \rSSsp{s} }^2 \d t \\
				\lesssim &  \| g \|_{ L^\infty_t ( \rSSs{s} ) }^2 \Nt{g}^2_{ \sumsp{s} } \lesssim  \Nt{g}^4_{ \sumsp{s} } \, .
			\end{align*}
			Similarly, one proves using \cite[Lemma 4.9]{GL24} that $\Nt{ \Lambda^\eps_\kin[g] }_{ \rSSSk{s} } \lesssim \Nt{ g }_{ \sumsp{s} }^2$, which implies \eqref{eq:fully_nonlin_bound}. The Lipschitz estimate \eqref{eq:fully_nonlin_Lip} is established in the same way.
		\end{proof}
		
		\subsection{Proof of Theorem \ref{thm:main}}
		\label{scn:proof_hydro_limit}
		
		Let us consider initial data $f_\ini$ with the size restriction
		$$\| f_\ini \|_{ \rSSs{s} \cap \rhSSsm{1-\boldsymbol{\alpha}} } \le \delta \ll 1$$
		as well as the corresponding global Navier-Stokes-Solution $f_\ns$ in the sense that it satisfies \eqref{eq:integral_NSF}:
		$$f_\ns(t) = U_\ns(t) f_\ini + \Psi_\ns[ f_\ns , f_\ns ](t) \, .$$
		We look for a solution of the form
		$$f^\eps = f^\eps_\kin + f^\eps_\disp + f_\ns + g^\eps$$
		where the kinetic initial layer generated by the microscopic part of the initial data $f^\eps_\kin(t) = U^\eps_\kin(t) f_\ini^\perp$ is such that by \cite[Lemma 4.8]{GL24} for some $\sigma > 0$
		$$\sup_{t \ge 0} e^{2 \sigma t / \eps^2} \| f^\eps_\kin(t) \|^2_{ \rSSs{s} } + \frac{1}{\eps^2} \int_0^\infty e^{2 \sigma t / \eps^2} \| f^\eps_\kin(t) \|^2_{ \rSSsp{s} } \d t \lesssim \| f_\ini^\perp \|_{ \rSSs{s }}^2 \ ,$$
		and the oscillating part $f^\eps_\disp(t) = U^\eps_\disp(t) f_{\ini, \ip} = U^\eps_\disp(t) f_{\ini, \ip}$ is generated by the the ill-prepared macroscopic part of the initial data, it satisfies by \cite[Lemma 4.5]{GL24}
		$$\| U^\eps_\disp(t) f_{\ini, \ip} \|_{ W^{s, \infty}_x (\Ssp_v) } \lesssim \left( \frac{\eps}{t} \right)^{ \frac{d-1}{2} } \| f_{\ini, \ip} \|_{ \dot{\mathbb{B}}_{1,1}^{\frac{d+1}{2} + s} (\Ssm_v) } \ .$$
		The new unknown $g^\eps$ is then to be sought in the space~$\sumsp{s}_\eps$ and satisfies the integral equation
		\begin{align*}
			g^\eps = & \eps \Lambda^\eps[f^\eps] + \Psi^\eps[ f^\eps ] - \Psi^\eps[f_\ns] + \left( U^\eps_\hyd f_\ini - U^\eps_\Wave f_\ini + \Psi^\eps[f_\ns] - f_\ns \right) \\
			& + U^\eps_\kin \Pi f_\ini + (U_\Wave^\eps-U^\eps_\disp) f_\ini \\
			= & \eps \Lambda^\eps[f^\eps] + \Psi^\eps[ f^\eps ] - \Psi^\eps[f_\ns] + \sum_{k=0}^{3} \SS_k^\eps
		\end{align*}
		where, using the integral formulation \eqref{eq:integral_NSF} of $f_\ns$, we denoted (note that $\SS^\eps_1$ is the same one as in \cite[eq. (2.15b)]{GL24})
		\begin{align*}
			\SS_1^\eps = & U^\eps_\hyd f_\ini - U^\eps_\Wave f_\ini + \Psi^\eps_\hyd[f_\ns] - f_\ns \\
			= & ( U^\eps_\ns - U_\ns ) f_\ini + ( \Psi^\eps_\ns-\Psi_\ns )[f_\ns]
		\end{align*}
		$$\SS_0^\eps = \Psi^\eps_\kin[f_\ns] \, , \qquad \SS^\eps_2 = U^\eps_\kin \Pi f_\ini \ , \qquad \SS^\eps_3 = ( U^\eps_\Wave - U^\eps_\disp ) f_\ini \ .$$
		We now expand the difference between the bilinear terms:
		\begin{align*}
			\Psi^\eps[f^\eps]& -\Psi^\eps[f_\ns] = \Psi^\eps( f^\eps_\kin + f^\eps_\disp + g^\eps ,  2 f_\ns + f^\eps_\kin + f^\eps_\disp + g^\eps ) \\
			&= \Psi^\eps[g^\eps] + 2 \Psi^\eps \left[g^\eps, f^\eps_\kin + f^\eps_\disp + f_\ns \right] + \Psi^\eps\left[ f^\eps_\kin + f^\eps_\disp , 2 f_\ns + f^\eps_\kin + f^\eps_\disp \right] \\
			&=:  \Psi^\eps[g^\eps] + \Phi^\eps g^\eps + \SS_4^\eps \, ,
		\end{align*}
		where $\Phi^\eps$ is a linear operator (depending on $f_\ini$). Denoting $\SS^\eps = \sum_{k=0}^{4} \SS^\eps_k$, we are thus left with
		\begin{equation}
			\label{eq:equation_g_eps}
			g^\eps = \SS^\eps + \Phi^\eps g^\eps + \Psi^\eps[g^\eps] + \eps \Lambda^\eps[f^\eps] \, .
		\end{equation}

		\subsubsection{Estimate of the source terms}
		The first part of the source term was estimated in \cite[Lemma 6.2]{GL24}. By tracking the intermediate estimates in its proof, because of the smallness assumption $\| \Pi f_\ini \|_{ \H^s } \le \delta$, one has
		$$\Nt{\SS^\eps_1}_{ \rSSSh{s} } \lesssim \eps^{r-s} \| f_\ini \|_{ \rSSs{r} \cap \rhSSsm{1-\boldsymbol{\alpha}} } \ .$$
		The second part of the source term is estimated using the fact that from \cite[Definition 2.4]{GL24}
		$$U^\eps_\kin = U^\eps_\kin \left( \id - P^\eps_\hyd \right)$$
		and that from \cite[Lemma 3.4]{GL24} (see also the proof of \cite[eq. (4.2)]{GL24} for more details)
		$$\forall r \in [s, s+1] \ , \quad \| P^\eps_\hyd - \Pi \|_{ \rSSs{r} \to \rSSsp{s} } \lesssim \eps^{r-s} \, .$$
		Therefore, there holds
		$$\Nt{ \SS_2^\eps }_{ \rSSSh{s} } = \Nt{ U^\eps_\kin \left( \Pi - P^\eps_\hyd \right) \Pi f_\ini }_{ \rSSSh{s} } \lesssim \eps^{r-s} \| \Pi f_\ini \|_{ \rSSs{r} } \, .$$
		The third term of the source term is estimated using \cite[Lemma 4.3]{GL24}:
		$$\forall r \in [s, s+1] \ , \quad \Nt{ \SS_3^\eps }_{ \rSSSk{s} } \lesssim \eps^{r-s} \| f_\ini \|_{  \rSSs{r} \cap \rhSSsm{1-\boldsymbol{\alpha}} } \ .$$
		The fourth term of the source term is estimated by expanding it as follows:
		$$\SS^\eps_4 = 2 \Psi^\eps[ f^\eps_\kin , f_\ns ] + \Psi^\eps[ f^\eps_\kin ] +2  \Psi^\eps[ f^\eps_\kin, f^\eps_\disp ] + 2 \Psi^\eps[ f^\eps_\disp , f_\ns ] + \Psi^\eps[f^\eps_\disp ]= \sum_{k=1}^{5} \SS^\eps_{4, k}$$
		Using \eqref{eq:bilinear_psi_kin_est} followed by \eqref{eq:linear_semigroup_est} and \eqref{eq:nsf_hydro_regime_est} we have the following estimate:
		\begin{align*}
			\sum_{k=1}^3 \Nt{ \SS^\eps_{4, k} }_{ \sumsp{s} } & \lesssim \eps \Nt{ f^\eps_\kin }_{ \rSSSk{s} } \left( \| \Pi f_\ini \|_{ \rSSs{r} \cap \rhSSsm{1-\boldsymbol{\alpha}} } + \| f_\ns \|_{ \rSSSh{s} } \right) \\
			& \lesssim \eps \| f_\ini^\perp \|_{ \rSSs{s} } \| \Pi f_\ini \|_{ \rSSs{s} \cap \rhSSsm{1-\boldsymbol{\alpha}}  } 
		\end{align*}
		and using \eqref{eq:convergence_psi_oscillation} followed by \eqref{eq:nsf_hydro_regime_est} and \eqref{eq:linear_semigroup_est}, we also have
		$$\sum_{k=4}^5 \Nt{ \SS^\eps_{4, k} }_{ \sumsp{s} } \lesssim \beta( f_{\ini, \ip} , \eps ) \| \Pi f_\ini \|_{ \rSSs{s} }$$
		so that together
		$$\Nt{ \SS^\eps_4 }_{ \sumsp{s} } \lesssim \left( \eps \| f_\ini \|_{ \rSSs{s} \cap \rhSSsm{1-\boldsymbol{\alpha}} } + \beta( f_{\ini, \ip} , \eps ) \right) \| f_\ini \|_{ \rSSs{s}} \, .$$
		In conclusion, using that $\| f_\ini \|_{\rSSs{s} \cap \dot{\H}_x^{1-\boldsymbol{\alpha}} ( \Ssm_v )} \le \delta$, we have the vanishing estimate
		\begin{equation}
			\label{eq:summary_source_terms}
			\forall r \in [s, s+1] \ , \quad \Nt{ \SS^\eps - \SS^\eps_0 }_{ \sumsp{s} } \lesssim  \eps^{r-s} \| f_\ini \|_{ \rSSs{r} \cap \rhSSsm{1-\boldsymbol{\alpha}}} + \beta( f_{\ini, \ip} , \eps ) \, .
		\end{equation}
		Including the last term (which we will prove to vanish in a weaker norm), we finally have using \eqref{eq:bilinear_psi_est} and \eqref{eq:nsf_hydro_regime_est}
		\begin{equation}
			\label{eq:source_term_est_crude}
			\Nt{ \SS^\eps }_{ \sumsp{s} } \lesssim \| f_\ini \|_{ \rSSs{s} } + \| f_\ini \|_{ \rhSSsm{1-\boldsymbol{\alpha}}  } \, .
		\end{equation}
		
		\subsubsection{Estimate of the linear term} Using \eqref{eq:bilinear_psi_est} followed by \eqref{eq:linear_semigroup_est} and \eqref{eq:nsf_hydro_regime_est}, we have the continuity estimate
		\begin{equation}
			\label{eq:linear_phi_est}
			\Nt{ \Phi^\eps }_{ \sumsp{s} \to \sumsp{s} } \lesssim \| f_\ini \|_{\rSSs{s}} + \| f_\ini \|_{ \rhSSsm{1-\boldsymbol{\alpha}}  } \, .
		\end{equation}
		
		\subsubsection{Proof of existence}
		Let us consider some small ball $\UU$ of $\sumsp{s}$ of radius $R$, where we will take $\delta \ll R \ll 1$, and the mapping
		$$\forall h \in \UU \, , \quad \Xi^\eps(h) = \SS^\eps + \Phi^\eps h + \Psi^\eps[h] + \eps \Lambda^\eps[ f^\eps_\kin + f^\eps_\disp + f_\ns + h ] \, .$$
		Using \eqref{eq:source_term_est_crude} for $\SS^\eps$, \eqref{eq:bilinear_psi_est} for $\Psi^\eps$, \eqref{eq:linear_phi_est} for $\Phi^\eps$ and \eqref{eq:fully_nonlin_bound} for $\Lambda^\eps$, we have the stability estimate
		$$\sup_{h \in \UU } \Nt{ \Xi^\eps( h ) }_{ \sumsp{s} } \lesssim \delta + \delta R + R^2  \, ,$$
		thus, for $\delta \ll R \ll 1$ we indeed have that $\Xi( \UU ) \subset \UU$. Similarly, using this time \eqref{eq:fully_nonlin_Lip} for $\Lambda^\eps$, there holds
		$$\Nt{ \Xi^\eps(h_1) -\Xi^\eps(h_2) }_{ \sumsp{s}} \lesssim (\delta + R) \Nt{ h_1 - h_2 }_{ \sumsp{s} } $$
		so that the same smallness assumption makes $\Xi$ a contraction. There thus exists a unique fixed point $g^\eps$, or in other words, a solution to \eqref{eq:equation_g_eps}, which is such that $\Nt{ g^\eps }_{ \sumsp{s} } \le R$ uniformly in $\eps$.

		\subsubsection{Proof of uniqueness}
		
		Consider two solutions $f_j^\eps$ ($j=1,2$) such that
		$$\sup_{t \ge 0} \| f^\eps_j(t) \|_{ \rSSs{s} } \le 2 \delta \quad \text{and} \quad |\nabla|^{ \boldsymbol{\alpha} } f^\eps_j \in L^2_{loc} \left( [0, \infty) ; \rSSsp{s - \boldsymbol{\alpha} } \right) \, .$$
		Their difference $h^\eps := f^\eps_1 - f^\eps_2$ satisfies the equation
		$$\partial_t h^\eps = \frac{1}{\eps^2} \left( \LL - \eps v \cdot \nabla_x \right) h^\eps + \frac{1}{\eps} \QQ\left( h^\eps , f^\eps_1 + f^\eps_2 \right) +  \left( \RR^\eps[ f^\eps_1] - \RR^\eps[ f^\eps_2 ] \right)$$
		thus, using \cite[eq. (5.1)]{GL24} to estimate $\QQ$ and \ref{AsN_Lip} for $\RR^\eps$, a simple energy estimate in $\rSSs{s}$ yields for some $C, \lambda > 0$ (that may change from one line to another)
		\begin{align*}
			\frac{1}{2} & \frac{\d}{\d t} \| h^\eps(t) \|_{ \rSSs{s} }^2 + \frac{\lambda}{\eps^2} \| h^\eps(t) \|_{ \rSSsp{s} }^2 \le \frac{C}{\eps^2} \| h^\eps(t) \|_{ \rSSs{s} }^2 \\
			& + C \| h^\eps (t) \|_{ \rSSsp{s} } \left( \left\| (f_1^\eps, f_2^\eps) \right\|_{\rSSs{s}} \left\| | \nabla_x |^{\boldsymbol{\alpha}} h^\eps \right\|_{ \rSSsp{s - \boldsymbol{\alpha}} } 
			+ \left\| | \nabla_x |^{\boldsymbol{\alpha}} (f^\eps_1, f^\eps_2) \right\|_{ \rSSsp{s-\boldsymbol{\alpha}} } \left\| h^\eps \right\|_{ \rSSs{s} }  \right)
		\end{align*}
		and therefore, by Young's inequality
		\begin{align*}
			\frac{1}{2} \frac{\d}{\d t} \| h^\eps(t) \|_{ \rSSs{s} }^2 & + \left( \frac{\lambda}{\eps^2} - C \delta^2 \right) \| h^\eps(t) \|_{ \rSSsp{s} }^2 \\
			& \le C \left( \frac{1}{\eps^2} + \left\| | \nabla_x |^{\boldsymbol{\alpha}} (f^\eps_1, f^\eps_2) \right\|_{ \rSSsp{s-\boldsymbol{\alpha}} }^2 \right) \left\| h^\eps \right\|_{ \rSSs{s} }^2  \, ,
		\end{align*}
		which simplifies for $\delta$ small enough as
		\begin{align*}
			\frac{1}{2} \frac{\d}{\d t} \| h^\eps(t) \|_{ \rSSs{s} }^2 & + \frac{\lambda}{\eps^2} \| h^\eps(t) \|_{ \rSSsp{s} }^2 \\
			& \le C \left( \frac{1}{\eps^2}  + \left\| | \nabla_x |^{\boldsymbol{\alpha}} (f^\eps_1, f^\eps_2) \right\|_{ \rSSsp{s-\boldsymbol{\alpha}} }^2 \right) \left\| h^\eps \right\|_{ \rSSs{s} }^2  \, .
		\end{align*}
		Because of the local integrability assumption on $|\nabla_x|^{\boldsymbol{\alpha}} f_j^\eps$, we conclude by a Gronwall argument that $f^\eps_1(t) = f^\eps_2(t)$ for $t$ close to zero, and then repeat the argument to conclude that $f^\eps_1 = f^\eps_2$.
		
		\subsubsection{Proof of convergence}
		Since $g^\eps$ solves the equation \eqref{eq:equation_g_eps}, one has that
		\begin{align*}
			\left(g^\eps - \SS^\eps_0\right)  & - \Phi^\eps \left(g^\eps - \SS^\eps_0 \right) - \Psi^\eps\left[g^\eps - \SS^\eps_0 , g^\eps \right] \\
			= & \SS^\eps - \SS^\eps_0 + \eps \Lambda^\eps[ f^\eps] + \Phi^\eps \SS^\eps_0  + \Psi^\eps\left[ g^\eps , \SS^\eps_0 \right]  \, ,
		\end{align*}
		and therefore the estimate for some $C > 0$
		$$\big[ 1 - C (\delta + R)  \big] \Nt{ g - \SS^\eps_0 }_{ \sumsp{s} } \le \Nt{ \text{r.h.s.} }_{ \sumsp{s} } \, ,$$
		or, more simply for $\delta$ small enough
		$$\Nt{ g - \SS^\eps_0 }_{ \sumsp{s} } \le \Nt{ \text{r.h.s.} }_{ \sumsp{s} } \, .$$
		Since we have constructed $g^\eps$ so that $\Nt{g^\eps}_{ \sumsp{s} } \le \delta$, we have by the bilinear estimate \eqref{eq:bilinear_from_kinetic}
		$$\Nt{ \SS^\eps_0 }_{ \rSSSk{s} } \lesssim \| f_\ns \|_{ \rSSSh{s} } \lesssim  \| f_\ini \|_{ \rSSs{s} \cap \rhSSsm{1-\boldsymbol{\alpha}} } \le \delta \, ,$$
		and therefore, using \eqref{eq:bilinear_psi_est} and \eqref{eq:linear_phi_est}, we get the vanishing estimate
		$$\Nt{\Phi^\eps \SS^\eps_0 + \Psi^\eps\left[ g^\eps , \SS^\eps_0 \right]}_{ \sumsp{s}} \lesssim \eps \| f_\ini \|_{ \rSSs{s} \cap \rhSSsm{1-\boldsymbol{\alpha}} } \ .$$
		Similarly, because of \eqref{eq:nsf_hydro_regime_est} and \eqref{eq:linear_semigroup_est}, we deduce that $\Nt{ f^\eps }_{\sumsp{\eps}} \lesssim \| f_\ini \|_{ \rSSs{s} \cap \rhSSsm{1-\boldsymbol{\alpha}} } $ and therefore by the estimate \eqref{eq:fully_nonlin_bound} for $\Lambda^\eps$
		$$\Nt{\eps \Lambda^\eps[h]}_{\rSSs{s}} \lesssim \eps \| f_\ini \|_{ \rSSs{s} \cap \rhSSsm{1-\boldsymbol{\alpha}} } \ .$$
		Combined with \eqref{eq:summary_source_terms} for $\SS^\eps - \SS^\eps_0$, we obtain
		$$\forall r \in [s, s+1] \ , \quad \Nt{ g - \SS^\eps_0 }_{ \sumsp{s} } \lesssim \eps^{r-s} \| f_\ini \|_{ \rSSs{r} \cap \rhSSsm{1-\boldsymbol{\alpha}}} + \beta( f_{\ini, \ip} , \eps ) \ .$$
		Finally, it was established in \cite[Section 6.3]{GL24} that
		$$\left\| \SS^\eps_0 \right\|_{ L^\infty_t \left( \rSSs{s} \right) } \lesssim \eps \Nt{ f_\ns }_{ \rSSSh{s} }^2 \, ,$$
		therefore, we have for $r \in [s, s+1]$
		\begin{align*}
			\left\| g^\eps \right\|_{ L^\infty_t \left( \rSSs{s} \right) } & \lesssim \eps \Nt{ f_\ns }_{ \rSSs{s} }^2 + \eps^{r-s} \| f_\ini \|_{ \rSSs{r} } + \beta( f_{\ini, \ip} , \eps ) \\
			& \lesssim \eps^{r-s} \| f_\ini \|_{ \rSSs{r} \cap \rhSSsm{1-\boldsymbol{\alpha}} } + \beta( f_{\ini, \ip} , \eps ) \, .
		\end{align*}
		As in the original proof \cite[Section 6]{GL24}, one may rely on a density argument to deduce that $\| g^\eps \|_{ L^\infty_t (\rSSs{s})}$ vanishes (unquantitatively) in the case $r=s$.

	\section{Application to the non-linear Fokker-Planck equation}
	\label{scn:NFP}

		Recall the definition of the nonlinear Fokker-Planck operator:
		$$\CC[F] = \nabla_v \cdot \left( T \nabla_v F + (v-U) F \right)$$
		where the macroscopic fields are given by
		$$R = \int F \ \d v \ , \quad U = \int v F \ \d v \qquad \text{and} \qquad T = \frac{1}{R}  \int |v-U|^2 F \ \d v \ .$$
		Consider the expansion of the particle density
		$$F = M + \eps f \ , \qquad M(v) = (2 \pi)^{-d/2} e^{-|v|^2/2} \ ,$$
		and of the nonlinear operator:
		$$\CC[F] = \eps \LL f + \eps^2 \QQ(f, f) + \eps^3 \RR^\eps [f] \ .$$
		The goal of this section is to provide usable expressions of $\LL$, $\QQ$ and $\RR$ and to prove that they satisfy Assumptions \ref{AsL}, \ref{AsB}, \ref{AsN}. To this end, we also recall the expression of the macroscopic fields associated to $f$:
		$$\rho = \int f \d v \, , \qquad u = \int v f \d v \, , \qquad \theta = \frac{1}{d} \int f \left( |v|^2-d \right) \d v \, .$$
		We will denote by $[f]_n$ any linear combination of terms of the form
		$$\frac{\rho^a u^b \theta^c \eps^\delta}{(1+\eps \rho)^p} \ , \qquad a, c, p, \delta \in \N \ , \quad b \in \N^d \ , \quad a + |b| + c \ge n \ . $$
		
		\subsection{Expansion of the collision operator}
		
		Recall the expansion of the temperature derived in Section \ref{sec:expansion_temperature}:
		$$T = 1 + \eps \theta - \eps^2 \left( \rho \theta + \frac{ |u|^2 }{d} \right) + \eps^3 [f]_3 \ .$$
		On the one hand
		\begin{align*}
			T \nabla_v F =&  \left( 1 + \eps \theta - \eps^2 \left( \rho \theta + \frac{|u|^2}{d} \right) + \eps^3 [f]_3 \right) \left( -v M + \eps \nabla_v f \right) \\
			= & -v M + \eps \left( \nabla_v f - v \theta M \right) + \eps^2 \left( v M \left( \rho \theta + \frac{|u|^2}{d}  \right) + \theta \nabla_v f \right) \\
			& + \eps^3 \left( [f]_3 v M + [f]_2 \nabla_v f \right) \ .
		\end{align*}
		On the other hand, recalling the expansion
		$$U = \eps u - \eps^2 \rho u +  \eps^3 [f]_3 \ ,$$
		there also holds
		\begin{align*}
			(v-U) F = & \left( v - \eps u + \eps^2 \rho u - \eps^3 [f]_3 \right) (M + \eps f) \\ 
			= & v M + \eps \left( v f - u M \right) + \eps^2 u \left( \rho M - f \right)  + \eps^3 \left( [f]_3 M + [f]_2 f \right) \ .
		\end{align*}
		Combining the two we obtain
		$$\CC[F] = \eps \LL f + \eps^2 \QQ(f, f) + \eps^3 \RR^\eps[f]$$
		where each part is given by
		\begin{gather}
			\label{eq:def_NFP_L_first}
			\LL f= \nabla_v \cdot \left( \nabla_v f +v f - \left( u + \theta v \right) M \right) \, , \\
			\notag
			\QQ(f, f) = (\theta \nabla_v - u) \cdot \nabla_v (f-\rho M) - \frac{|u|^2}{d} (|v|^2-d) M \, , \\
			\label{eq:structure_R_eps}
			\RR^\eps[f] = [f]_3 \cdot \left(1, v, v^{\otimes 2} \right) + [f]_2 \cdot (\nabla_v f , \nabla^2_v f ) \, .
		\end{gather}
		Note that, since $\CC[F](v) \in \ker(\LL)^\perp$ for any suitable $F$, it is also the case for the coefficients of its Taylor expansion around $M$, namely $\LL$, $\QQ$ and $\RR^\eps$:
		\begin{equation}
			\label{eq:NFP_orthogonal_expansion}
			\forall \varphi \in \ker(\LL) \, , \quad \la \LL f , \varphi \ra_\Ss = \la \QQ ( f , f ) , \varphi \ra_\Ss = \la \RR^\eps [f ], \varphi \ra_\Ss = 0 \, .
		\end{equation}

		\subsection{Functional framework}

		The weight function to be considered is the Maxwellian $\mu = M$ and the regular space the corresponding weighted $H^1$ space:
		$$\Ss = L^2 \left( M^{-1}(v) \d v \right) \quad \text{and} \quad \Ssp := H^1\left( M^{-1}(v) \d v \right) \, .$$
		We introduce the following anisotropic gradient
		$$\natv f := \nabla_v f + v f = M \nabla_v \left( M^{-1} f \right) = - \nabla_v^* f \, ,$$
		where the adjoint is taken with respect to the inner product of $\Ss$. One checks by an integration by parts that the graph-norm of $\natv$ is equivalent to the $\Ssp$--norm:
		\begin{align}
			\notag
			(d+1) \| f \|^2_{ \Ss } + \| \natv f \|_{ \Ss }^2 & = (d+1) \| f \|^2_{ \Ss } + \| \nabla_v f \|^2_\Ss + \| v f \|^2_\Ss + 2 \la v f , \nabla_v f \ra_\Ss \\
			\label{eq:anisotropic_gradient_H1}
			& = \| f \|^2_{ \Ss } + \| \nabla_v f \|^2_\Ss = \| f \|^2_{\Ssp} \, ,
		\end{align}
		furthermore, by Young's inequality
		$$- \| \nabla_v f \|^2_\Ss + \frac12 \| v f \|^2_\Ss \le \| \widetilde{\nabla}_v f \|^2_\Ss \le 2 \| vf \|^2_\Ss + 2 \| \nabla_v f \|^2_\Ss \ ,$$
		thus, interpolating these two lower bounds, we have the estimate
		\begin{align*}
			(d+1) \| f \|^2_\Ss + \| \widetilde{\nabla}_v f \|^2_\Ss & \ge \frac34 \| f \|_{\Ssp}^2 + \frac14 \left( \frac12 \|vf\|^2_\Ss - \| \nabla_v f\|^2_\Ss \right) \\
			& \ge \frac18 \left( \| f \|^2_{\Ssp} + \| vf \|^2_{\Ss} \right) \ .
		\end{align*}
		We thus have the comparison
		\begin{equation}
			\label{eq:characterization_weighted_H1v}
			\| f \|_{\Ssp} \approx \| f \|_{\Ss} + \| \natv f \|_{\Ss} \approx  \| \la v \ra f \|_{\Ss} + \| \nabla_v f \|_{\Ss} \, .
		\end{equation}
		Finally, recall the gaussian Poincaré inequality:
		$$\int_{\R^d} \varphi(v) M(v) d v = 0 \Rightarrow \int_{\R^d} | \nabla_v \varphi(v) |^2 M(v) \d v$$
		which, letting $f = \varphi M$, rewrites
		\begin{equation}
			\label{eq:gaussian_poincare}
			\int f \d v = 0 \Rightarrow \| f \|_{L^2(M^{-1})}^2 \le  \| \natv f \|_{ L^2(M^{-1}) }^2  \, ,
		\end{equation}
		thus for zero-mean functions one actually has
		\begin{equation}
			\label{eq:characterization_weighted_H1v_zero_mean}
			\| f \|_{\Ssp} \approx \| \natv f \|_{\Ss} \approx  \| v f \|_{\Ss} + \| \nabla_v f \|_{\Ss} \, .
		\end{equation}
		\subsection{Checking the linear assumptions}
		
		\label{scn:NFP_linear}
		Let us check that $\LL$ satisfies the linear assumptions \ref{AsL}. Note that \ref{L2} is indeed satisfied since $\mu=M$.
		
		\subsubsection{Assumption \ref{L1}}
		Note that the linear operator writes
		$$\LL f = \nabla_v \cdot \left( M \nabla_v \left( \frac{f}{M} \right) \right) - \nabla_v \cdot \left( M \nabla_v \left( u \cdot v + \frac{\theta}{2} (|v|^2-d)  \right) \right)$$
		which allows to rewrite it under a structurally more suitable form:
		\begin{equation}
			\label{eq:def_L_NFP}
			\LL f = \nabla_v \cdot \left( M \nabla_v \left( \frac{f^\perp}{M} \right)\right) = - \natv^* \cdot \natv f^\perp \, .
		\end{equation}
		Furthermore, since the range of $\LL$ is orthogonal to its null-space by \eqref{eq:NFP_orthogonal_expansion}, one even has
		$$\LL = \left[ \widetilde{\nabla} (\id - \Pi) \right]^* \left[ \widetilde{\nabla} (\id - \Pi) \right] \, ,$$
		so that it is easy to see that $\LL$ is self-adjoint with domain $H^2(M^{-1} \d v)$. Finally, $\LL$ commute with orthogonal change of variables since multiplication by $v_i^n$, derivation and integration do. 
		
		\subsubsection{Assumption \ref{L3}}
		
		Since any element $f \in \ker(\LL)^\perp$ has zero mean ($\rho_f=0$) one has in virtue of \eqref{eq:characterization_weighted_H1v_zero_mean} for some $\lambda_\LL > 0$
		\begin{equation}
			\label{eq:NFP_strong_coercivity}
			\la \LL f , f \rangle_{\Ss} = - \| \widetilde{\nabla} f^\perp \|_{\Ss}^2 \le - \lambda_\LL \| f^\perp \|^2_{\Ssp} \, .
		\end{equation}
		
		\subsubsection{Assumption \ref{L4}}
		To find a decomposition $\LL = \BB + \AA$, one typically lets $\Ss_j = L^2\left( \mu^{-1} \la v \ra^j \right)$ and proves that $\left[ \LL , \la v \ra^{ j} \right]$ is of lower order with respect to $\LL ( \la v \ra^{j} \cdot )$. To reduce this problem, we use the commutator identity $[AB, C] = A [B, C] + [A, C] B$:
		\begin{align*}
			\left[ \LL , \la v \ra^{j} \right] & = \left[ \nabla_v \cdot \natv (\id - \Pi) , \la v \ra^{j} \right] \\
			& = - \nabla_v \cdot \natv \left[ \Pi , \la v \ra^{j} \right] + [ \nabla_v \cdot \natv , \la v \ra^{j} ] (\id - \Pi) \, .
		\end{align*}
		On the one hand, there holds
		$$\forall a, b \in \N \, , \quad \left\| \la v \ra^a \nabla_v^b \left[ \Pi , \la v \ra^{j} \right] \varphi \right\|_{ \Ss } \le C_{a,b} \| \varphi \|_{\Ss} \, ,$$
		as well as
		\begin{align*}
			[ \nabla_v \cdot \natv , \la v \ra^{j} ] & = [ \Delta_v + v \cdot \nabla_v,  \la v \ra^{j} ] \\
			& = \Delta_v \la v \ra^{j} + 2 \nabla_v \la v \ra^{j} \cdot \nabla_v + v \cdot \nabla_v \la v \ra^{j} \\
			& = \OO\left( \la v \ra^j \right) + \OO\left( \la v \ra^{j-1} \right) \nabla_v \, ,
		\end{align*}
		so that, put together, there holds for any $0 < \eta \ll 1$:
		\begin{align}
			\notag
			\la [ \LL , \la v \ra^j] f , \la v \ra^j f \ra_{\Ss} \le &  C_\eta \| \la v \ra^j f \|^2_{\Ss} + \eta \| \nabla_v f \|^2_{ \Ss } \\
			\label{eq:A}
			\le & C_\eta' \| \la v \ra^j f \|^2_{\Ss} + \eta \| \nabla_v ( \la v \ra^j  f) \|^2_{ \Ss } \, .
		\end{align}
		On the other hand, the coercivity estimate \eqref{eq:NFP_strong_coercivity} and the norm equivalence \eqref{eq:characterization_weighted_H1v} together imply that for some $\lambda > 0$
		$$\la \LL \left( \la v \ra^j f \right) , \la v \ra^j f \ra_{\Ss} \le -\lambda \left( \| \nabla_v (\la v \ra^j f)^\perp \|^2_{\Ss} + \| \la v \ra (\la v \ra^j f)^\perp \|^2_{\Ss} \right)$$
		and since both $ \la v \ra \Pi $ and $\nabla_v \Pi$ are bounded on $\Ss$, we deduce
		\begin{equation}
			\label{eq:B}
			\la \LL \left( \la v \ra^j f \right) , \la v \ra^j f \ra_{\Ss}   \le - \lambda \left( \|  \la v \ra^{j+1} f \|^2_{\Ss} + \|  \nabla_v \left( \la v \ra^j f \right)\|^2_\Ss \right) +C  \| \la v \ra^j f \|^2_{ \Ss }  \, .
		\end{equation}
		Taking $\eta$ small enough, the controls \eqref{eq:A} and \eqref{eq:B} combine for some $C, \lambda>0$ as
		$$\la \LL f , f \la v \ra^{2 j} \ra_{ \Ss } \le - \left( \lambda \| \la v \ra^{j+1} f \|_{ \Ss }^2 - C \| \la v \ra^j f \|^2_{\Ss} \right)  \ ,$$
		therefore, for any arbitrary parameter $R > 0$
		\begin{align*}
			\la \LL f , f \la v \ra^{2 j} \ra_{ \Ss } 
			\le & - \left( \lambda \la R \ra^2 - C \right) \| \mathbf{1}_{|v| \ge R} \la v \ra^{j} f \|_{ \Ss }^2  + ( \lambda \la R \ra^2 + C ) \| \mathbf{1}_{|v| \le R} \la v \ra^j f \|_{ \Ss }^2 \\
			\le & - \left( \lambda \la R \ra^2 - C \right) \| \la v \ra^{j} f \|_{ \Ss }^2  + 2 ( \lambda \la R \ra^2 + C ) \| \mathbf{1}_{|v| \le R} \la v \ra^j f \|_{ \Ss }^2 \, .
		\end{align*}
		We therefore define
		$$\BB := \LL - \AA \ , \quad \AA := 2 ( \lambda \la R \ra^2 + C ) \mathbf{1}_{|v| \le R} \quad \text{and} \quad \la R \ra^2 \ge ( C + \lambda_\LL ) / \lambda$$
		so as to obtain the dissipativity estimate
		\begin{align*}
			\la \BB f , f \la v \ra^{2 j} \ra_{ \Ss } & \le - \lambda_\LL \| \la v \ra^j \|_{\Ss}^2 \, .
		\end{align*}
		This then implies \ref{assumption_dissipative} with $\lambda_\BB = \lambda_\LL$ since $i v \cdot \xi$ is skew-symmetric.
		
		\subsection{Nonlinear estimates}
		
		The orthogonality assumption \ref{Bortho} was already justified in \eqref{eq:NFP_orthogonal_expansion} and the isotropy one \ref{Bisotrop} for $\QQ$ follows the same justification as the one for $\LL$.
		
		\subsubsection{Assumption \ref{Bbound}}
		The symmetrized expression of $\QQ$ is
		\begin{align*}
			2 \QQ(f, f') = & (\theta \nabla_v - u) \cdot \nabla_v (f'-\rho' M) + (\theta' \nabla_v - u') \cdot \nabla_v (f-\rho M) \\
			& - 2\frac{u \cdot u'}{d} (|v|^2-d) M \ ,
		\end{align*}
		therefore, by the characterization \eqref{eq:characterization_weighted_H1v} of $\Ssp$, one has that
		$$\| \QQ(f, g)  \|_{ \Ssm } \lesssim \|  f \|_{ \Ss } \left\| g \right\|_{ \Ssp }  \, .$$
		
		\subsubsection{Assumptions \ref{AsN_bound} and \ref{AsN_Lip}}
		Recall from \eqref{eq:structure_R_eps} that for some $A_n(\rho, u, \theta) = [f]_n$, there holds
		\begin{equation*}
			\RR^\eps[f] = A_2^\eps(\rho, u, \theta) \cdot (\nabla_v f , \nabla^2_v f) + A^\eps_3(\rho, u, \theta) \cdot \left(1, v, |v|^2 \right) M \ .
		\end{equation*}
		Using the definition of $\Ssm$ followed by the bilinear estimate of Lemma \ref{lem:bilinear_estimate_sobolev}, we have for any $0 < \boldsymbol{\alpha} < d/2 < s$
		\begin{align*}
			\| A_2^\eps(\rho, u, \theta) \cdot (\nabla_v f , \nabla^2_v f) \|_{ \rSSsm{s} \cap \rhSSsm{ \boldsymbol{\alpha}-1 } } \lesssim & \| A_2^\eps(\rho, u, \theta) \cdot ( f , \nabla_v f ) \|_{ \rSSs{s} \cap \rhSSs{ \boldsymbol{\alpha}-1 } } \\
			\lesssim & \| A_2^\eps(\rho, u, \theta) \|_{ \H^s } \left\| | \nabla_x|^{\boldsymbol{\alpha}} ( f , \nabla_v f ) \right\|_{ \rSSs{s-\boldsymbol{\alpha}} } \\
			\lesssim & \| f \|_{ \rSSs{s} } \left\| | \nabla_x|^{\boldsymbol{\alpha}} f \right\|_{ \rSSsp{s-\boldsymbol{\alpha}} }
		\end{align*}
		where we used \eqref{eq:modulo_estimate} to bound $A_2^\eps$ in the last line since $\| f \|_{ \rSSs{s} } \le \delta \ll 1$. Similarly, one has
		\begin{align*}
			\| A_2^\eps & (\rho, u, \theta) \cdot (\nabla_v f , \nabla_v^2 f)  - A_2^\eps(\rho', u', \theta') \cdot (\nabla_v f' , \nabla_v^2 f') \|_{ \rSSsm{s} \cap \rhSSsm{ \boldsymbol{\alpha}-1 } } \\
			\lesssim & \| \left( A_2^\eps(\rho, u, \theta) - A_2^\eps(\rho', u', \theta') \right) \cdot ( f , \nabla_v f ) \|_{ \rSSs{s} \cap \rhSSs{ \boldsymbol{\alpha}-1 } } \\
			& + \| A_2^\eps(\rho', u', \theta') \cdot ( f - f' , \nabla_v (f-f') ) \|_{ \rSSs{s} \cap \rhSSs{ \boldsymbol{\alpha}-1 } } \\
			\lesssim & \| f-f' \|_{ \rSSs{s} } \left\| | \nabla_x|^{\boldsymbol{\alpha}} f \right\|_{ \rSSsp{s-\boldsymbol{\alpha}} } + \| f' \|_{ \rSSs{s} } \left\| | \nabla_x|^{\boldsymbol{\alpha}} (f-f') \right\|_{ \rSSsp{s-\boldsymbol{\alpha}} } \ .
		\end{align*}
		The estimate for the second term of $\RR^\eps$ is simpler, so we have indeed
		$$\| \RR^\eps[f] \|_{ \rSSsm{s} \cap \rhSSsm{ \boldsymbol{\alpha}-1 } } \lesssim \| f \|_{ \rSSs{s} } \| | \nabla_x |^{\boldsymbol{\alpha}} f \|_{ \rSSsp{s-\boldsymbol{\alpha}} } $$
		\begin{align*}
			\| \RR^\eps[f] - \RR^\eps[f'] \|_{ \rSSsm{s} \cap \rhSSsm{ \boldsymbol{\alpha}-1 } } \lesssim & \| f-f' \|_{ \rSSs{s} } \| | \nabla_x |^{\boldsymbol{\alpha}} f \|_{ \rSSsp{s-\boldsymbol{\alpha}} } \\
			& + \| f' \|_{ \rSSs{s} } \| | \nabla_x |^{\boldsymbol{\alpha}} (f-f') \|_{ \rSSsp{s-\boldsymbol{\alpha}} } \ .
		\end{align*}
		
	\section{The case of the BGK equation}
	\label{scn:BGK}
	
	Recall the definition of the BGK collision operator
	$$\CC[ F ] = \MM[F] - F \ , \qquad \MM[F] = \frac{R}{\sqrt{T}} M\left( \frac{v-U}{\sqrt{T}} \right)$$
	where $M$ is the centered reduced gaussian and $(R, U, F)$ are the moments of $F$ defined for instance at the beginning of Section \ref{scn:NFP}. By the expansion of $\MM[F]$ from Section \ref{scn:exp_maxw}, there holds
	$$\CC[ M + \eps f] = \eps \LL f + \eps^2 \QQ(f,f) + \eps^3 \RR^\eps[f]$$
	where the linear part is (recall the definition \eqref{eq:def_Pi} of $\Pi$)
	$$\LL f = - f^\perp = -(\id - \Pi ) f \ ,$$
	and the bilinear part is
	\begin{align*}
		M^{-1} \QQ(f,f) = & - \left( \frac{d(d+2)}{8} \theta^2  + \frac{|u|^2}{2} \right) - \frac{d+2}{2} (v \cdot u) \theta  \\
		& - \frac{|v|^2-d}{2} \left( \frac{d+2}{2} \theta^2  + \frac{|u|^2}{d} \right) + \frac12 \left( v \cdot u + \theta \frac{|v|^2}{2} \right)^2 \ .
	\end{align*}
	The fully nonlinear remainder writes as a linear combination of terms of the form
	$$\{f\}_3 I \left( \eps \{f\}_1 \cdot V_2 \right)^\alpha J \left( \eps \{f\}_1 \right)^\beta \cdot V_6 \ , \qquad \alpha, \beta \in \{0,1\} \ ,$$
	where we used the following notations of Appendix \ref{scn:exp_maxw}:
	\begin{itemize}
		\item $[f]_n$ denotes any linear combination of terms of the form
		$$\frac{\rho^a u^b \theta^c \eps^\delta}{(1+\eps \rho)^p} \ , \qquad a, c, p, \delta \in \N \ , \quad b \in \N^d \ , \quad a + |b| + c \ge n \ , $$
		\item $\{f\}_n$ denotes any linear combination of terms $\displaystyle \frac{[f]_n}{ (1+\eps[f]_1)^q }$ with $q \in \N$.
		\item $V_n = (1, v, v^{\otimes 2} , \dots, v^{\otimes n})$.
		\item $\displaystyle I(r) = \sum_{ \ge 0} \frac{(-r)^n}{(3+n)!}$ and $\displaystyle J(r) = \sum_{n \ge 0} \binom{d/2}{3+n}$.
	\end{itemize}
	
	\subsection{Functional framework}

	We define the weight function as $\mu = M$ and the spaces
	$$\Ss = \Ssp = \Ssm =  L^2( \mu^{-1} \d v ) \, .$$

	\subsection{Estimates on the linear operator}
	
	Since $\LL = -( \id - \Pi)$ where
	$$\forall a, b \ge 0 \, , \quad \la v \ra^{a} \left| \Pi \left( \la v \ra^b f \right)  \right| \le C_{a, b} \| f \|_{ \Ss } \, , $$
	it is easy to check that Assumptions \ref{AsL} hold by adapting the arguments for the nonlinear Fokker-Planck operator from Section \ref{scn:NFP_linear}.
	
	\subsection{Estimates on the bilinear operator}
	
	Note that $\la v \ra^n M \in \Ss$ for any $n \ge 0$ and the bilinear operator $\QQ(f, f)$ is a linear combination of terms of the form
	$$\rho^a u^b \theta^c v^\alpha M \ , \quad \text{where} \quad a, c \in \N \ , \quad b, \alpha \in \N^d \ , \quad a + |b| +c = 2 \ , $$
	so it follows that
	$$| \QQ(f,f)(v) | \lesssim \| \Pi f \|_{\Ss}^2 M(v)^{3/4} \qquad \text{and thus} \qquad \| \QQ(f, f) \|_{ \Ss } \lesssim \| \Pi f \|_{\Ss}^2 \ ,$$
	which is stronger than Assumption \ref{AsB}.
	
	\subsection{Estimates on the fully nonlinear remainder}
	
	Since $\| \{f\}_1 \|_{\H^s} \lesssim \delta$ by \eqref{eq:modulo_estimate}, taking $\delta$ small enough, the product estimate involving analytic functions of Lemma \ref{lem:bilinear_analytic_sobolev}, we have for some $C>0$
	\begin{align*}
		\| \{f\}_1 & M(v) I \left( \eps \{f\}_1 \cdot V_2 \right)^\alpha J \left( \eps \{f\}_1 \right)^\beta \cdot V_6 \|_{ \H^s } \\
		& \lesssim \delta \JJ\left( C \delta \right) \II\left( C \delta \la v \ra^2 \right) e^{-\frac{|v|^2}{2}} \la v \ra^6 \ ,
	\end{align*}
	with the notation $\II, \JJ$ from Lemma \ref{lem:bilinear_analytic_sobolev}. Using the estimate for $\{f\}_n$ from \eqref{eq:modulo_estimate} and the fact that $\II(r) \lesssim e^{r}$, we then deduce
	\begin{align*}
		\| \{f\}_1 & I \left( \eps \{f\}_1 \cdot V_2 \right)^\alpha J \left( \eps \{f\}_1 \right)^\beta \cdot V_6 \|_{ \H^s } \\
		& \lesssim  M(v) e^{-\frac{|v|^2}{2} + C \delta \la v \ra^2} \la v \ra^6 \ ,
	\end{align*}
	and thus, for $\delta$ small
	\begin{align*}
		\| \{f\}_1 I \left( \eps \{f\}_1 \cdot V_2 \right)^\alpha J \left( \eps \{f\}_1 \right)^\beta \cdot V_6 \|_{ \H^s } \lesssim  M(v)^{3/4} \ .
	\end{align*}
	Considering $\{f\}_3 = \{f\}_1 \{f\}_2$, we conclude thanks to Lemma \ref{lem:bilinear_estimate_sobolev} that
	$$\| \RR^\eps[f] \|_{ \rSSs{s} \cap \rhSSs{\boldsymbol{\alpha}-1} } \lesssim \| f \|_{ \rSSs{s} } \| |\nabla_x|^{\boldsymbol{\alpha}} f \|_{ \rSSs{s-\boldsymbol{\alpha}} } \ .$$
	One proves similarly that
	\begin{align*}
		\| \RR^\eps[f] - \RR^\eps[f'] \|_{ \rSSs{s} \cap \rhSSs{\boldsymbol{\alpha}-1} } \lesssim & \| f -f' \|_{ \rSSs{s} } \| |\nabla_x|^{\boldsymbol{\alpha}} (f, f') \|_{ \rSSs{s-\boldsymbol{\alpha}} } \\
		& + \| (f,f') \|_{ \rSSs{s} } \| |\nabla_x|^{\boldsymbol{\alpha}} (f- f') \|_{ \rSSs{s-\boldsymbol{\alpha}} } \ .
	\end{align*}
	
	\section{The Boltzmann-Fermi-Dirac equation}
	\label{scn:BFD}
	
	The estimates we point out in this section are taken from \cite{JXZ22}. Note that they were established for $\gamma=1$, however they hold for any $\gamma \in [0, 1]$ by replacing the use of the triangle inequality
	$$|v-v_*| \lesssim \la v \ra \la v_* \ra \qquad \text{by} \qquad |v-v_*|^\gamma \lesssim \la v \ra^\gamma \la v_* \ra^\gamma$$
	and noting that it still holds as a classical result of the theory for the Boltzmann equation that
	$$\nu(v) = \int_{ \R^d_{ v_* } \times \S^{d-1}_\sigma } B(v-v_*, \sigma) \MM_* \d \sigma \d v_* \approx \la v \ra^\gamma \, .$$
	We expand the Boltzmann-Fermi-Dirac collision operator as
	$$\CC[ \MM + \eps f ] = \eps \LL f + \eps^2 \QQ(f, f) + \eps^3 \RR[f]$$
	where $\LL$ is defined in \cite[(1.11)]{JXZ22}, $\QQ$ in \cite[(1.15)]{JXZ22} and $\RR[f] = \TT(f, f, f)$ with $\TT$ defined in \cite[(1.16)]{JXZ22}.
	
	We define the weight $\mu$ as
	$$\mu = \MM (1 - \MM)$$
	and the functional spaces as
	$$\Ss = L^2(\mu^{-1} ) \, , \qquad \Ssp = L^2\left( \la v \ra^\gamma \mu^{-1} \right) \, , \qquad \Ssm = L^2\left( \la v \ra^{-\gamma} \mu^{-1} \right) \, .$$
	The assumption \ref{L3} comes from \cite[Proposition 2.1]{JXZ22}, \ref{Bbound} comes from \cite[eq. (2.6)]{JXZ22}, the sufficient assumption \ref{AsN_suff} is given by \cite[eq. (2.8)]{JXZ22}.
	
	\begin{rem}
		The quantum Landau equation does not fall within our framework: the nonlinear operator involves terms of the form
		$$\nabla \cdot \left( F(v)^2 \int_{ \R^d } a(v,v_*) (\nabla F) (v_*) \d v_* \right) \ ,$$
		thus its expansion around $\MM$ would involve in particular a quadratic term $\alpha(v) f^2(v)$, which makes it impossible to establish closed estimates in $L^2_v$ (without working in $H^{\frac{d}{2}+\eps}_v$). This equation is usually considered together with an \textit{a priori} bound in $L^\infty$ corresponding to Pauli's exclusion principle, which is not really compatible with the hilbertian setting of the present document.
		
		The quantum Boltzmann equation, however, falls within our framework because of its stronger non-local nature: the nonlinear operator only involves terms of the form
		$$\int_{ \S^{d-1} \times \R^d } a(v,v_*, \sigma) F( u_1 ) F( u_2 ) F( u_3 ) \d \sigma \d v_* \ ,$$
		where each $u_k$ is a function of $(v, v_*)$ and there never holds $u_k = u_\ell$ for $k \neq \ell$.
	\end{rem}
	
	\appendix

	\section{Expansion of macroscopic quantities}

\label{scn:exp_macro}
	In this section, we establish expansions at order $\OO(\eps^3)$ of the macroscopic quantities
	\begin{gather*}
		R = \int F \d v \, , \qquad
		U = \frac{1}{R} \int F v \,  \d v \, , \qquad
		T = \frac{1}{d R} \int F | v - U |^2 \, \d v \, ,
	\end{gather*}
	as well as the usual Maxwellian with the same moments as $F$:
	\begin{equation*}
		\MM[ F ] = \MM[R ; U ; T] = \frac{R}{ (2 \pi T)^{ \frac{d}{2} } } \exp \left( - \frac{|v - U |^2}{2 T} \right) = \frac{R}{ T^{d/2} } M \left(  \frac{v-U}{\sqrt{T}}\right) \ ,
	\end{equation*}
	where $M$ denotes the centered reduced gaussian in dimension $d$. Let us consider the expansion around $M$
	\begin{equation*}
		F^\eps = M + \eps f \ ,
	\end{equation*}
	we denote the macroscopic fluctuations associated to $f$ as
	\begin{gather*}
		\rho = \int f \d v \, , \qquad u = \int v f \d v \, , \qquad \theta = \frac{1}{d} \int f \left( |v|^2-d \right) \d v \, .
	\end{gather*}
	The subsequent computations will correspond to the case $\mu = M$ in Assumption \ref{AsL}, therefore $(E, K) = (d, 1+2/d)$ and
	\begin{equation}
		\label{eq:Pi_Maxwellian}
		M(v)^{-1} \Pi f(v) = \rho + u \cdot v + \theta \frac{|v|^2-d}{2} \ .
	\end{equation}
	\textbf{Notation:} For every quantity $A^\eps$ we expand, we will denote for any $n \ge 0$
	$$A^\eps = \sum_{k=0}^{n-1} \eps^k A_k + \eps^{n} A_{n}^\eps \, .$$
	We will denote by $[f]_n$ any linear combination of terms of the form
	$$\frac{\rho^a u^b \theta^c \eps^\delta}{(1+\eps \rho)^p} \ , \qquad a, c, p, \delta \in \N \ , \quad b \in \N^d \ , \quad a + |b| + c \ge n \ , $$
	and by $\{f\}_n$ any linear combination of terms of the form
	$$\frac{[f]_n}{ (1+\eps[f]_1)^q } \ , \qquad q \in \N \ . $$
	Note that, for any $s > \frac{d}{2}$, Lemma \ref{lem:bilinear_analytic_sobolev} implies that for any $\NN(\rho, u, \theta) = [f]_n, \{f\}_n$ with $n \ge 2$, there is some $\delta_\NN > 0$ small enough such that for any $\|(\rho, \rho', u,u', \theta, \theta') \|_{ \H^s } \le \delta_\NN$
	\begin{subequations}
		\label{eq:modulo_estimate}
			\begin{gather}
			\| \NN(\rho,u,\theta) \|_{ \H^s \cap \dot{\H}^{\boldsymbol{\alpha}-1} } \lesssim \| \nabla^{\boldsymbol{\alpha}} (\rho,u,\theta) \|_{ \H^{s-\boldsymbol{\alpha}} } \| (\rho,u,\theta) \|_{ \H^s } \ , 
		\end{gather}
		and, since $\NN$ has the structure of rational function, on may also prove the Lipschitz estimate
		\begin{equation}
			\begin{split}
				\| \NN(\rho,u,\theta) & - \NN(\rho',u',\theta') \|_{ \H^s \cap \dot{\H}^{\boldsymbol{\alpha}-1} } \\
				& \lesssim \| \nabla^{\boldsymbol{\alpha}} (\rho,u,\theta) - \nabla^{\boldsymbol{\alpha}} (\rho',u',\theta') \|_{ \H^{s-\boldsymbol{\alpha}} } \| (\rho, \rho',u, u',\theta, \theta') \|_{ \H^s } \ .
			\end{split}
		\end{equation}
	\end{subequations}
	Finally, we will denote
	$$V_n = (1, v, v^{\otimes 2} , \dots, v^{\otimes n}) \ .$$
	
	\subsection{Expansion of the density and velocity}

	We have the expansions
	$$R^\eps = 1 + \eps \rho \, , \qquad U^\eps = \eps u  - \eps^2 \rho u + \eps^3 [f]_3 \, .$$
	
	\subsection{Expansion of the temperature}
	\label{sec:expansion_temperature}
	We start with the definition of $T^\eps$ and use the fact that $\int |v|^2 F \d v = d + d \eps (\theta + \rho) = d R^\eps + d \eps \theta$:
	\begin{align*}
		T^\eps = & \frac{1}{d R^\eps } \int \left( | v |^2 + | U^\eps |^2  - 2 v \cdot U^\eps \right) F \d v \\
		= & 1 + \frac{\eps \theta}{R^\eps} - \frac{|U^\eps|^2}{d} \\
		= & 1 + \eps \theta - \eps^2 \left( \rho \theta + \frac{ |u|^2 }{d} \right) + \eps^3 [f]_3 \, .
	\end{align*}
	
	\subsection{Expansion of the centered energy density}
	One has the expansion
	\begin{align*}
		W^\eps = \frac{1}{2} | v - U^\eps |^2 = & \frac{|v|^2}{2} + \frac{| U^\eps |}{2} - v \cdot U^\eps \\
		= & \frac{|v|^2}{2} - \eps u \cdot v + \eps^2 \left( \frac{|u|^2}{2} + \rho u \cdot v \right) + \eps^3 [f]_3 \cdot (1, v) \ .
	\end{align*}
	
	\subsection{Expansion of the inverse temperature}
	One has the expansions
	$$\frac{1}{T^\eps} =1 - \eps T^\eps_1 + \eps^2 ( T^\eps_1 )^2 - \eps^3 \frac{( T^\eps_1 )^3}{T^\eps}$$
	therefore, since $T^\eps = 1 + \eps[f]_1$, one has
	$$\frac{1}{T^\eps} = 1 - \eps \theta + \eps^2 \left( \theta (\rho + \theta) + \frac{|u|^2}{d} \right) + \eps^3 \{f\}_3 \ .$$

	\subsection{Expansion of the rescaled relative energy density}
	\label{scn:exp_X}
	We now consider the quantity $X^\eps = \frac{|v-U^\eps|^2}{2 T^\eps}$. Multiplying the two previous expansions yields
	\begin{align*}
		X^\eps = \frac{W^\eps}{T^\eps} = & \frac{|v|^2}{2} - \eps \left( u \cdot v + \frac{|v|^2}{2} \theta \right) \\
		& + \eps^2 \left( (\rho+\theta) \left( u \cdot v + \theta \frac{|v|^2}{2} \right) + \frac{|u|^2}{2} \left( 1 + \frac{|v|^2}{d} \right) \right) \\
		& + \eps^3 \{f\}_3 \cdot V_2 \ .
	\end{align*}

	\subsection{Expansion of the unscaled Maxwellian}
 	Let us denote $I(r) = \sum_{n \ge 0} \frac{(-r)^n}{(3+n)!}$, we have that
	$$e^{ - X^\eps }  = e^{ - \frac{|v|^2}{2}} e^{-\eps X_1^\eps} =  e^{- \frac{|v|^2}{2} } \left( 1 - \eps X^\eps_1 + \frac{\eps^2}{2} (X_1^\eps)^2 + \eps^3 (X_1^\eps)^3 I(\eps X_1^\eps) \right)$$
	therefore, one has
	$$\frac{e^{ - X^\eps }}{ (2 \pi)^{d/2} } = M \left( 1 + \eps E_1 + \eps^2 E_2 + \eps^3 E_3^\eps \right)$$
	where the first terms are given by
	\begin{gather*}
		E_1 = \theta \frac{|v|^2}{2} + u \cdot v \ , \\
		E_2 = -\frac{|v|^2}{2} \left( \theta(\rho+\theta) + \frac{|u|^2}{d} \right) - (\rho+\theta)(u \cdot v) - \frac{|u|^2}{2} + \frac12 \left( u \cdot v + \theta \frac{|v|^2}{2} \right)^2 \ ,
	\end{gather*}
	and the remainer is of the form
	 $$E_3^\eps = \{f\}_3 \left[ 1 +  I \left( \eps \{f\}_1 \cdot V_2 \right) \right] \cdot V_6 \, .$$
	 
	 \subsection{Expansion of the $\boldsymbol{-\frac{d}{2}}$--power of the temperature}
	 Let us denote $J(r) = \sum_{n \ge 0} \binom{d/2}{n+3}$, we have the expansion
	 $$\left( T^\eps \right)^{-d/2} = 1 - \frac{d}{2} \eps T^\eps_1 + \frac{d}{4} \left(\frac{d}{2} + 1\right) \eps^2 ( T^\eps_1 )^2 + \eps^3 J(\eps\{ f \}_1 )$$
	 so we deduce
	 $$\left( T^\eps \right)^{-d/2} =  1 - \eps \frac{d}{2} \theta - \eps^2 \left(  \frac{|u|^2 }{2}  + \frac{d}{2} \rho \theta - \frac{d (d+2)}{8} \theta^2 \right) + \eps^3 \{f\}_3 \left[ 1 + J(\eps \{f\}_1) \right] \ .$$

	 \subsection{Expansion of the Maxwellian}
	 \label{scn:exp_maxw}
	 We now turn to the expansion of the Maxwellian:
	 \begin{align*}
	 	\MM[ M + \eps f ] = \frac{R^\eps e^{-X^\eps} }{ (2\pi T^\eps )^{d/2} } = M \frac{R^\eps e^{-\eps X^\eps_1}}{(T^\eps)^{ d/2 } } \ ,
	 \end{align*}
	 and thus one has
	 $$\MM[ M + \eps f ] = M + \eps \Pi f + \eps^2 \QQ(f,f) + \eps^3 \RR^\eps[f]$$
	 where $\Pi$ is defined in \eqref{eq:Pi_Maxwellian}, the bilinear approximation $\QQ(f,f)$ is given by
	 \begin{align*}
	 	M^{-1} \QQ(f,f) = &  - \left( |u|^2 - \frac{d(d+2)}{8} \theta^2  -\rho^2 \right) \\
	 	& + \frac{|v|^2}{2} \left( \frac{d+2}{2} \theta^2 + \frac{|u|^2}{2} \right) + \frac{d+2}{2} \theta (u \cdot v) \\
	 	& - \frac12 \left( u\cdot v +\theta \frac{|v|^2}{2} \right)^2 \ ,
	 \end{align*}
	 and the fully nonlinear remainder $\RR^\eps[f]$ is a linear combination of terms of the form
	 $$M(v)  \{f\}_3 I \left( \eps \{f\}_1 \cdot V_2 \right)^\alpha J \left( \eps \{f\}_1 \right)^\beta \cdot V_6 \ , \qquad \alpha, \beta \in \{0,1\} \ .$$
	 As for \eqref{eq:modulo_estimate}, one proves that
	 
	 \section{Nonlinear estimates in Sobolev spaces}
	 
	 \label{scn:nonlinear_sobolev}
	 \begin{lem}
	 	\label{lem:bilinear_estimate_sobolev}
	 	For any $0 \le \boldsymbol{\alpha} < \frac{d}{2} < s$, there holds for some $C = C(s, \boldsymbol{\alpha}) > 0$
	 	$$\| \psi \varphi \|_{ \H^s \cap \dot{\H}^{\boldsymbol{\alpha}-1}} \le C \| \psi \|_{ \H^s } \| | \nabla_x |^{\boldsymbol{\alpha}} \varphi \|_{ \H^{s-\boldsymbol{\alpha} } }$$
	 \end{lem}
	 
	 \begin{proof}
	 	First, note that because $0 \le \boldsymbol{\alpha} < \frac{d}{2} < s$, one has
	 	$$\| \varphi \|_{L^\infty} \lesssim \| \widehat{\varphi} \|_{L^1} \lesssim \| | \nabla_x |^{\boldsymbol{\alpha}} \varphi \|_{ \H^{s-\boldsymbol{\alpha} } } \quad \text{since} \quad \left( \la \xi \ra^{s-\boldsymbol{\alpha}} | \xi |^{\boldsymbol{\alpha}} \right)^{-1} \in L^2_\xi  \, .$$
	 	As a consequence, using \cite[Corollary 2.54]{BCD}, it follows that
	 	\begin{align*}
	 		\| \psi \varphi  \|_{ \dot{\H}^s } \lesssim & \| \psi \|_{ \H^s } \left( \| \varphi \|_{L^\infty} + \| \varphi \|_{ \dot{\H}^s} \right) \\
	 		\lesssim & \| \psi \|_{ \H^s } \| | \nabla_x |^{\boldsymbol{\alpha}} \varphi \|_{ \H^{s-\boldsymbol{\alpha} } } \, ,
	 	\end{align*}
	 	and thus
	 	\begin{equation}
	 		\label{eq:bilinear_sobolev}
	 		\| \psi \varphi  \|_{ \H^s } \lesssim \| \psi \|_{ \H^s } \| | \nabla_x |^{\boldsymbol{\alpha}} \varphi \|_{ \H^{s-\boldsymbol{\alpha} } } \, .
	 	\end{equation}
	 	Furthemore, one has from \cite[Corollary 2.55]{BCD} in the case of $d=2$
	 	$$ \| \varphi \psi \|_{ \dot{\H}^{\boldsymbol{\alpha}-1} } \lesssim \| \psi \|_{L^2} \| \varphi \|_{ \dot{\H}^{ \boldsymbol{\alpha} } } \lesssim \| \psi \|_{ \H^s } \| | \nabla_x |^{\boldsymbol{\alpha}} \varphi \|_{ \H^{s-\boldsymbol{\alpha} } }  \, ,$$
	 	and when $d \ge 3$, this control follows from \eqref{eq:bilinear_sobolev} since $0 \le \boldsymbol{\alpha} - 1 \le s$.
	 	Put together, these estimates yield the conclusion.
	 \end{proof}
	 
	 \begin{lem}
	 	\label{lem:bilinear_analytic_sobolev}
	 	For any $s > \frac{d}{2}$, there exists some $C=C(s) > 0$ such that the following holds. Consider $K \ge 0$ analytic functions given by
	 	$$A_k(z) = \sum_{n \ge 0} a_{k,n} z^n  \, , \qquad |z| < R \le \infty \ , \qquad k = 1, \dots, K \ ,$$
	 	then under the smallness assumption $\| \psi_k \|_{ \H^s } < R / C$, one has the estimates:
	 	$$\left\| \psi \prod_{k=1}^K A_k( \psi_k ) \right\|_{ \H^s  } \le \| \psi \|_{ \H^s } \prod_{k=1}^K \AA_k ( C \| \psi_k \|_{ \H^s } )  $$
	 	 where we have denoted
	 	$$\AA_k(z) := \sum_{n \ge 0} |a_{k,n}| z^{n} \, .$$
	 \end{lem}
	 
	 \begin{proof}
	 	By expanding each $A_k$, one has (writing $\boldsymbol{n} = (n_1, \dots, n_K)$)
	 	\begin{align*}
	 		\left\| \varphi \prod_{k=1}^K A_k( \psi_k ) \right\|_{ \H^s } = & \left\| \varphi \prod_{k=1}^K \sum_{ n \ge 0 } a_{k, n}  \psi_k^{n} \right\|_{ \H^s } \\
	 		= & \left\| \sum_{ \boldsymbol{n} \in \N^K } \varphi \prod_{k=1}^K a_{k, n_k} \psi_k^{n_k} \right\|_{ \H^s } \, , 
	 	\end{align*}
	 	thus, since $s > \frac{d}{2}$, denoting $C = C(s) > 0$ the operator norm of the multiplication in $\H^s$
	 	\begin{align*}
	 		\left\| \varphi \prod_{k=1}^K A_k( \psi_k ) \right\|_{ \H^s } \le & \| \varphi \|_{ \H^s } \sum_{ \boldsymbol{n} \in \N^K } \prod_{k=1}^K |a_{k, n_k}|  \left( C \| \psi_k \|_{ \H^s } \right)^{n_k} \\
	 		= & \| \varphi \|_{ \H^s } \prod_{k=1}^K \sum_{n \ge 0} | a_{k, n}|  \left( C \| \psi_k \|_{ \H^s } \right)^{n} \\
	 		= &  \| \varphi \|_{ \H^s } \prod_{k=1}^K \AA_k \left( C \| \psi_k \|_{ \H^s } \right) \, .
	 	\end{align*}
	 	This concludes the proof.
	 \end{proof}
	 
	 \begin{prop}
	 	\label{prop:AsN_mult_suff}
	 	Assumption \ref{AsN_suff} implies \ref{AsN_bound} and \ref{AsN_Lip}.
	 \end{prop}
	 
	 \begin{proof}
	 	To simplify the proof, we consider $N=3$.
	 	The assumption \ref{AsN_bound} can be proved using that $\RR$ is trilinear and local in $x$ through Littlewood-Paley decompositions. We denote for any $f=f(x,v)$ for which the computations make sense
	 	$$f = \sum_{j \in \Z} f_j \, , \qquad f_j = \chi(2^{-j} D_x ) f$$
	 	where $\chi$ is some appropriate cutoff function supported in some annulus $\AA$ (see for instance \cite[Section 2.2]{BCD}). Recall that as a classical result, there exists some integer $K > 0$ such that
	 	$$ j \le k \le \ell \Rightarrow \text{supp} \left( \widehat{ f_j g_k h_\ell } \right) \subset \bigcup_{ | m - \ell | \le K } 2^m \AA$$
	 	and thus one has
	 	$$(f g h)_m = \sum_{ | \max\{ j, k, \ell \} - m | \le K } ( f_j g _k h_\ell )_m \, .$$
	 	Thanks to this observation, starting from \ref{AsN_suff}, we have
	 	\begin{align*}
	 		\left\| \RR(f, g, h)_j \right\|_{ L^2_x \Ssm_v } & \lesssim \sum_{ | \max\{ k, \ell, m \} - j | \le K } \left\| \RR\left( f_{k} , g_{\ell}, h_{m} \right) \right\|_{ L^2_x \Ssm_v } \\
	 		\lesssim & \sum_{ | \max\{ k, \ell, m \} - j | \le K } \left\| \| \overline{f}_{k} \|_{ \Ssp_v } \| \overline{g}_{\ell} \|_{ \Ss_v } \| \overline{h}_{m} \|_{ \Ss_v } \right\|_{ L^2_x } \\
	 		\lesssim & \left( \sum_{ | k - j | \le K }  \| \overline{f}_{k} \|_{ L^2_x \Ssp_v } \right) \left( \sum_{k \le j + K} \| \overline{g}_{k} \|_{ L^\infty_x \Ss_v } \right) \left( \sum_{k \le j + K} \| \overline{h}_{k} \|_{ L^\infty_x \Ss_v } \right) \\
	 		& + \left( \sum_{ | k - j | \le K }  \| \overline{f}_{k} \|_{ L^2_x \Ss_v } \right) \left( \sum_{k \le j + K} \| \overline{g}_{k} \|_{ L^\infty_x \Ssp_v } \right) \left( \sum_{k \le j + K} \| \overline{h}_{k} \|_{ L^\infty_x \Ss_v } \right) \, ,
	 	\end{align*}
	 	where, to simplify the subsequent computations, the right hand side is to be summed over all permutations $( \overline{f} , \overline{g} , \overline{h} )$ of $(f, g, h)$. Before we go any further, let us point out that for any $0 \le { \boldsymbol{\alpha} } < \frac{d}{2} < s$, there holds
	 	\begin{align}
	 		\notag
	 		\sum_{k \in \Z} \| \varphi_k \|_{ L^\infty_x Y_v  }
	 		& \lesssim \left( \sum_{ k \le 0 } 2^{ 2 k { \boldsymbol{\alpha} } } \| \varphi_k \|_{ L^2_x Y_v }^2 \right)^{ \frac12 }  + \left( \sum_{ k \ge 0 } 2^{ 2 k s } \| \varphi_k \|_{ L^2_x Y_v }^2 \right)^{ \frac12 } \\
	 		\label{eq:LP_sum_all}
	 		& \lesssim \left\| | \nabla_x |^{ \boldsymbol{\alpha} } \varphi \right\|_{ \H^{s-{ \boldsymbol{\alpha} }}_x Y_v  } \, , 
	 	\end{align}
	 	where we used Bernstein's inequality and the characterization of Sobolev norms. On the other hand, there holds uniformly in $j \le 0$
	 	\begin{align}
	 		\notag
	 		\sum_{k \le j} \| \varphi_k \|_{ L^\infty_x Y_v  } & \lesssim \left( \sum_{ k \le j} 2^{ k (d - 2 { \boldsymbol{\alpha} })  }  \right)^{\frac12} \left( \sum_{k \le 0} 2^{2 { \boldsymbol{\alpha} } k} \| \varphi_k \|_{ L^2_x Y_v }^2 \right)^{\frac12} \\
	 		\label{eq:LP_sum_low}
	 		& \lesssim 2^{ j \left( \frac{d}{2} - { \boldsymbol{\alpha} } \right) } \left\| | \nabla_x |^{ \boldsymbol{\alpha} } \varphi \right\|_{ \H^{s-{ \boldsymbol{\alpha} }}_x Y_v  }  \, .
	 	\end{align}
	 	Turning back to the estimate of $\RR(f, g, h)_j$, we have using \eqref{eq:LP_sum_all} with ${ \boldsymbol{\alpha} } = 0$
	 	\begin{align*}
	 		\left\| \RR(f, g, h)_j \right\|_{ L^2_x \Ssm_v }
	 		\lesssim & \| \overline{h} \|_{ \rSSs{s} } \left( \sum_{ | k - j | \le K }  \| \overline{f}_{k} \|_{ L^2_x \Ssp_v } \right) \left( \sum_{k \le j + K} \| \overline{g}_{k} \|_{ L^\infty_x \Ss_v } \right) \\
	 		& + \| \overline{h} \|_{ \rSSs{s} } \left( \sum_{ | k - j | \le K }  \| \overline{f}_{k} \|_{ L^2_x \Ss_v } \right) \left( \sum_{k \le j + K} \| \overline{g}_{k} \|_{ L^\infty_x \Ssp_v } \right) \, .
	 	\end{align*}
	 	We can then say that for $j \le 0$, one has from \eqref{eq:LP_sum_low} with ${ \boldsymbol{\alpha} }=0$ for the first line and $0 \le { \boldsymbol{\alpha} } < \frac{d}{2}$ for the second line
	 	\begin{equation}
	 		\label{eq:T_low}		
	 		\begin{split}
	 			\left\| \RR(f, g, h)_j \right\|_{ L^2_x \Ssm_v }
	 			\lesssim & \| \overline{h} \|_{ \rSSs{s} } \| \overline{g} \|_{ \rSSs{s} }  \sum_{ | k - j | \le K }  2^{k \frac{d}{2}} \| \overline{f}_{k} \|_{ L^2_x \Ssp_v } \\
	 			& + \| \overline{h} \|_{ \rSSs{s} } \| |\nabla_x|^{ \boldsymbol{\alpha} } \overline{g} \|_{ \rSSs{s-{ \boldsymbol{\alpha} }}}  \sum_{ | k - j | \le K }  2^{k \left( \frac{d}{2}-{ \boldsymbol{\alpha} } \right) } \| \overline{f}_{k} \|_{ L^2_x \Ss_v }  \, .
	 		\end{split}
	 	\end{equation}
	 	We can also say that for $j \ge 0$, one has from \eqref{eq:LP_sum_all} with ${ \boldsymbol{\alpha} }=0$ for the first line and $0 \le { \boldsymbol{\alpha} } < \frac{d}{2}$ for the second line
	 	\begin{equation}
	 		\label{eq:T_high}		
	 		\begin{split}
	 			\left\| \RR(f, g, h)_j \right\|_{ L^2_x \Ssm_v }
	 			\lesssim & \| \overline{h} \|_{ \rSSs{s} } \| \overline{g} \|_{ \rSSs{s} }  \sum_{ | k - j | \le K }   \| \overline{f}_{k} \|_{ L^2_x \Ssp_v } \\
	 			& + \| \overline{h} \|_{ \rSSs{s} } \| |\nabla_x|^{ \boldsymbol{\alpha} } \overline{g} \|_{ \rSSs{s-{ \boldsymbol{\alpha} }}}  \sum_{ | k - j | \le K }  \| \overline{f}_{k} \|_{ L^2_x \Ss_v }  \, .
	 		\end{split}
	 	\end{equation}
	 	This allows to finally conclude using \eqref{eq:T_low} and \eqref{eq:T_high} that
	 	\begin{align*}
	 		\| \RR& (f, g, h) \|_{ \dot{\H}^{ { \boldsymbol{\alpha} } - \frac{d}{2} }_x \Ssm_v }^2 + \| \RR(f,g,h) \|_{ \rSSsm{s} }^2 \\
	 		\lesssim &
	 		\sum_{j \le 0} 2^{2 j \left( { \boldsymbol{\alpha} } - \frac{d}{2} \right) } \| \RR(f, g, h)_{j} \|_{ L^2_x \Ssm_v }^2 + \sum_{j \ge 0} 2^{2 j s } \| \RR(f, g, h)_{j} \|_{ L^2_x \Ssm_v }^2 \\
	 		\lesssim & \| \overline{h} \|_{ \rSSs{s} }^2 \| \overline{ g } \|_{ \rSSs{s} }^2 \left( \sum_{j \le 0} 2^{2 j { \boldsymbol{\alpha} }} \| f_{j} \|_{ L^2_x \Ssp_v }^2 +  \sum_{j \ge 0} 2^{2 s j} \| f_{j} \|_{ L^2_x \Ssp_v }^2 \right) \\
	 		& + \| \overline{h} \|_{ \rSSs{s} }^2 \| | \nabla_x |^{ \boldsymbol{\alpha} } \overline{ g } \|_{ \rSSsp{s-{ \boldsymbol{\alpha} }} }^2 \left( \sum_{j \le 0} \| f_{j} \|_{ L^2_x \Ss_v }^2 +   \sum_{j \ge 0} 2^{2 s j} \| f_{j} \|_{ L^2_x \Ss_v }^2 \right) 
	 		\, ,
	 	\end{align*}
	 	which concludes to
	 	\begin{align*}
	 		\| \RR (f, g, h) \|_{ \dot{\H}^{ { \boldsymbol{\alpha} } - \frac{d}{2} }_x \Ssm_v } + \| \RR(f,g,h) \|_{ \rSSsm{s} }  \lesssim  \| |\nabla_x|^{ \boldsymbol{\alpha} } \overline{f} \|_{ \rSSs{s-{ \boldsymbol{\alpha} }}} \| \overline{g} \|_{ \rSSs{s} } \| \overline{h} \|_{ \rSSs{s}} \, .
	 	\end{align*}
	 \end{proof}
	 
	 \bibliographystyle{plain}
	\bibliography{bibli}

@book{BCD,
	author = {Bahouri, Hajer and Chemin, Jean-Yves and Danchin, Rapha{\"e}l},
	title = {Fourier analysis and nonlinear partial differential equations},
	fseries = {Grundlehren der Mathematischen Wissenschaften},
	series = {Grundlehren Math. Wiss.},
	issn = {0072-7830},
	volume = {343},
	isbn = {978-3-642-16829-1; 978-3-642-16830-7},
	year = {2011},
	publisher = {Berlin: Springer},
	language = {English},
	doi = {10.1007/978-3-642-16830-7},
	zbMATH = {5826218},
	Zbl = {1227.35004}
}

@article{JXZ22,
	author = {Jiang, Ning and Xiong, Linjie and Zhou, Kai},
	title = {The incompressible {Navier}-{Stokes}-{Fourier} limit from {Boltzmann}-{Fermi}-{Dirac} equation},
	fjournal = {Journal of Differential Equations},
	journal = {J. Differ. Equations},
	issn = {0022-0396},
	volume = {308},
	pages = {77--129},
	year = {2022},
	language = {English},
	doi = {10.1016/j.jde.2021.10.061},
	zbMATH = {7436209},
	Zbl = {1478.35164}
}

@article{GL24,
	author = {Gervais, Pierre and Lods, Bertrand},
	title = {Hydrodynamic limits for kinetic equations preserving mass, momentum and energy: a spectral and unified approach in the presence of a spectral gap},
	fjournal = {Annales Henri Lebesgue},
	journal = {Ann. Henri Lebesgue},
	issn = {2644-9463},
	volume = {7},
	pages = {969--1098},
	year = {2024},
	language = {English},
	doi = {10.5802/ahl.215},
	zbMATH = {7914810},
	Zbl = {1547.35548}
}

@book{LR23,
	author = {Lemarie-Rieusset, Pierre Gilles},
	title = {The {Navier}-{Stokes} problem in the 21st century},
	edition = {2nd edition},
	isbn = {978-0-367-48726-3; 978-1-032-62373-3; 978-1-003-04259-4},
	year = {2023},
	publisher = {Boca Raton, FL: CRC Press},
	language = {English},
	doi = {10.1201/9781003042594},
	zbMATH = {7775147},
	Zbl = {1539.76044}
}

@article{BGLI91,
	author = {Bardos, Claude and Golse, Fran{\c{c}}ois and Levermore, David},
	title = {Fluid dynamic limits of kinetic equations. {I}: {Formal} derivations},
	fjournal = {Journal of Statistical Physics},
	journal = {J. Stat. Phys.},
	issn = {0022-4715},
	volume = {63},
	number = {1-2},
	pages = {323--344},
	year = {1991},
	language = {English},
	doi = {10.1007/BF01026608},
	zbMATH = {8191333}
}

@article{BGLII93,
	author = {Bardos, Claude and Golse, Fran{\c{c}}ois and Levermore, C. David},
	title = {Fluid dynamic limits of kinetic equations. {II}: {Convergence} proofs for the {Boltzmann} equation},
	fjournal = {Communications on Pure and Applied Mathematics},
	journal = {Commun. Pure Appl. Math.},
	issn = {0010-3640},
	volume = {46},
	number = {5},
	pages = {667--753},
	year = {1993},
	language = {English},
	doi = {10.1002/cpa.3160460503},
	zbMATH = {567347},
	Zbl = {0817.76002}
}

@article{GT20,
	author = {Gallagher, Isabelle and Tristani, Isabelle},
	title = {On the convergence of smooth solutions from {Boltzmann} to {Navier}-{Stokes}},
	fjournal = {Annales Henri Lebesgue},
	journal = {Ann. Henri Lebesgue},
	issn = {2644-9463},
	volume = {3},
	pages = {561--614},
	year = {2020},
	language = {English},
	doi = {10.5802/ahl.40},
	zbMATH = {7249460},
	Zbl = {1448.35393}
}

@article{BU91,
	author = {Bardos, Claude and Ukai, Seiji},
	title = {The classical incompressible {Navier}-{Stokes} limit of the {Boltzmann} equation},
	fjournal = {M\(^3\)AS. Mathematical Models \& Methods in Applied Sciences},
	journal = {Math. Models Methods Appl. Sci.},
	issn = {0218-2025},
	volume = {1},
	number = {2},
	pages = {235--257},
	year = {1991},
	language = {English},
	doi = {10.1142/S0218202591000137},
	zbMATH = {32878},
	Zbl = {0758.35060}
}

@article {CGT26,
	AUTHOR = {Carrapatoso, Kleber and Gallagher, Isabelle and Tristani,
	Isabelle},
	TITLE = {The {N}avier--{S}tokes limit of kinetic equations for low
	regularity data},
	JOURNAL = {Tunis. J. Math.},
	FJOURNAL = {Tunisian Journal of Mathematics},
	VOLUME = {8},
	YEAR = {2026},
	NUMBER = {3},
	PAGES = {497--538},
	ISSN = {2576-7658,2576-7666},
	MRCLASS = {35Q20 (35Q30 76D05 76P05)},
	MRNUMBER = {5073100},
	DOI = {10.2140/tunis.2026.8.497},
	URL = {https://doi.org/10.2140/tunis.2026.8.497},
}

@article{GSR04,
	author = {Golse, Fran{\c{c}}ois and Saint-Raymond, Laure},
	title = {The {Navier}-{Stokes} limit of the {Boltzmann} equation for bounded collision kernels},
	fjournal = {Inventiones Mathematicae},
	journal = {Invent. Math.},
	issn = {0020-9910},
	volume = {155},
	number = {1},
	pages = {81--161},
	year = {2004},
	language = {English},
	doi = {10.1007/s00222-003-0316-5},
	zbMATH = {2078251},
	Zbl = {1060.76101}
}

@article{GSR09,
	author = {Golse, Fran{\c{c}}ois and Saint-Raymond, Laure},
	title = {The incompressible {Navier}-{Stokes} limit of the {Boltzmann} equation for hard cutoff potentials},
	fjournal = {Journal de Math{\'e}matiques Pures et Appliqu{\'e}es. Neuvi{\`e}me S{\'e}rie},
	journal = {J. Math. Pures Appl. (9)},
	issn = {0021-7824},
	volume = {91},
	number = {5},
	pages = {508--552},
	year = {2009},
	language = {English},
	doi = {10.1016/j.matpur.2009.01.013},
	zbMATH = {5558950},
	Zbl = {1178.35290}
}

@article{LM10,
	author = {Levermore, C. David and Masmoudi, Nader},
	title = {From the {Boltzmann} equation to an incompressible {Navier}-{Stokes}-{Fourier} system},
	fjournal = {Archive for Rational Mechanics and Analysis},
	journal = {Arch. Ration. Mech. Anal.},
	issn = {0003-9527},
	volume = {196},
	number = {3},
	pages = {753--809},
	year = {2010},
	language = {English},
	doi = {10.1007/s00205-009-0254-5},
	zbMATH = {5731026},
	Zbl = {1304.35476}
}

@article{A12,
	author = {Ars{\'e}nio, Diogo},
	title = {From {Boltzmann}'s equation to the incompressible {Navier}-{Stokes}-{Fourier} system with long-range interactions},
	fjournal = {Archive for Rational Mechanics and Analysis},
	journal = {Arch. Ration. Mech. Anal.},
	issn = {0003-9527},
	volume = {206},
	number = {2},
	pages = {367--488},
	year = {2012},
	language = {English},
	doi = {10.1007/s00205-012-0557-9},
	zbMATH = {6134304},
	Zbl = {1257.35140}
}

@article{G23,
	author = {Gervais, Pierre},
	title = {On the convergence from {Boltzmann} to {Navier}-{Stokes}-{Fourier} for general initial data},
	fjournal = {SIAM Journal on Mathematical Analysis},
	journal = {SIAM J. Math. Anal.},
	issn = {0036-1410},
	volume = {55},
	number = {2},
	pages = {805--848},
	year = {2023},
	language = {English},
	doi = {10.1137/22M1471687},
	zbMATH = {7683251},
	Zbl = {1512.35073}
}

@article{B15,
	author = {Briant, Marc},
	title = {From the {Boltzmann} equation to the incompressible {Navier}-{Stokes} equations on the torus: a quantitative error estimate},
	fjournal = {Journal of Differential Equations},
	journal = {J. Differ. Equations},
	issn = {0022-0396},
	volume = {259},
	number = {11},
	pages = {6072--6141},
	year = {2015},
	language = {English},
	doi = {10.1016/j.jde.2015.07.022},
	zbMATH = {6484234},
	Zbl = {1326.35220}
}

@article{BMM19,
	author = {Briant, Marc and Merino-Aceituno, Sara and Mouhot, Cl{\'e}ment},
	title = {From {Boltzmann} to incompressible {Navier}-{Stokes} in {Sobolev} spaces with polynomial weight},
	fjournal = {Analysis and Applications (Singapore)},
	journal = {Anal. Appl., Singap.},
	issn = {0219-5305},
	volume = {17},
	number = {1},
	pages = {85--116},
	year = {2019},
	language = {English},
	doi = {10.1142/S021953051850015X},
	zbMATH = {7007695},
	Zbl = {1405.35134}
}

@article{CC26,
	author = {Cao, Chuqi and Carrapatoso, Kleber},
	title = {Hydrodynamic limit for the non-cutoff {Boltzmann} equation},
	fjournal = {Annales de l'Institut Henri Poincar{\'e} C. Analyse Non Lin{\'e}aire},
	journal = {Ann. Inst. Henri Poincar{\'e} C, Anal. Non Lin{\'e}aire},
	issn = {0294-1449},
	volume = {43},
	number = {2},
	pages = {417--482},
	year = {2026},
	language = {English},
	doi = {10.4171/aihpc/139},
	zbMATH = {8191825}
}

@article{EP75,
	author = {Ellis, Richard S. and Pinsky, Mark A.},
	title = {The first and second fluid approximation to the linearized {Boltzmann} equation},
	fjournal = {Journal de Math{\'e}matiques Pures et Appliqu{\'e}es. Neuvi{\`e}me S{\'e}rie},
	journal = {J. Math. Pures Appl. (9)},
	issn = {0021-7824},
	volume = {54},
	pages = {125--156},
	year = {1975},
	language = {English},
	zbMATH = {3449148},
	Zbl = {0286.35062}
}

@article{YY16,
	author = {Yang, Tong and Yu, Hongjun},
	title = {Spectrum analysis of some kinetic equations},
	fjournal = {Archive for Rational Mechanics and Analysis},
	journal = {Arch. Ration. Mech. Anal.},
	issn = {0003-9527},
	volume = {222},
	number = {2},
	pages = {731--768},
	year = {2016},
	language = {English},
	doi = {10.1007/s00205-016-1010-2},
	zbMATH = {6647626},
	Zbl = {1348.35164}
}

@article{R21,
	author = {Rachid, Mohamad},
	title = {Incompressible {Navier}-{Stokes}-{Fourier} limit from the {Landau} equation},
	fjournal = {Kinetic and Related Models},
	journal = {Kinet. Relat. Models},
	issn = {1937-5093},
	volume = {14},
	number = {4},
	pages = {599--638},
	year = {2021},
	language = {English},
	doi = {10.3934/krm.2021017},
	zbMATH = {7450819},
	Zbl = {1476.35165}
}

@article {SR03,
	AUTHOR = {Saint-Raymond, Laure},
	TITLE = {From the {BGK} model to the {N}avier-{S}tokes equations},
	JOURNAL = {Ann. Sci. \'Ecole Norm. Sup. (4)},
	FJOURNAL = {Annales Scientifiques de l'\'Ecole Normale Sup\'erieure.
	Quatri\`eme S\'erie},
	VOLUME = {36},
	YEAR = {2003},
	NUMBER = {2},
	PAGES = {271--317},
	ISSN = {0012-9593},
	MRCLASS = {76D05 (35Q30 76A02 76P05 82C40)},
	MRNUMBER = {1980313},
	MRREVIEWER = {J\"urgen\ Socolowsky},
	DOI = {10.1016/S0012-9593(03)00010-7},
	URL = {https://doi.org/10.1016/S0012-9593(03)00010-7},
}

@article {CJ24,
	AUTHOR = {Choi, Young-Pil and Jung, Jinwook},
	TITLE = {Incompressible {N}avier-{S}tokes limit from nonlinear
	{V}lasov-{F}okker-{P}lanck equation},
	JOURNAL = {Appl. Math. Lett.},
	FJOURNAL = {Applied Mathematics Letters. An International Journal of Rapid
	Publication},
	VOLUME = {158},
	YEAR = {2024},
	PAGES = {Paper No. 109214, 7},
	ISSN = {0893-9659,1873-5452},
	MRCLASS = {76N17 (35Q35 76P05)},
	MRNUMBER = {4774079},
	DOI = {10.1016/j.aml.2024.109214},
	URL = {https://doi.org/10.1016/j.aml.2024.109214},
}

@article{CRT22,
	author = {Carrapatoso, Kleber and Rachid, Mohamad and Tristani, Isabelle},
	title = {Regularization estimates and hydrodynamical limit for the {Landau} equation},
	fjournal = {Journal de Math{\'e}matiques Pures et Appliqu{\'e}es. Neuvi{\`e}me S{\'e}rie},
	journal = {J. Math. Pures Appl. (9)},
	issn = {0021-7824},
	volume = {163},
	pages = {334--432},
	year = {2022},
	language = {English},
	doi = {10.1016/j.matpur.2022.05.009},
	zbMATH = {7541871},
	Zbl = {1491.35310}
}

@article{MM16,
	author = {Mathiaud, J. and Mieussens, L.},
	title = {A {Fokker}-{Planck} model of the {Boltzmann} equation with correct {Prandtl} number},
	fjournal = {Journal of Statistical Physics},
	journal = {J. Stat. Phys.},
	issn = {0022-4715},
	volume = {162},
	number = {2},
	pages = {397--414},
	year = {2016},
	language = {English},
	doi = {10.1007/s10955-015-1404-9},
	zbMATH = {6566272},
	Zbl = {1334.35352}
}

@incollection{V02,
	author = {Villani, C{\'e}dric},
	title = {A review of mathematical topics in collisional kinetic theory.},
	booktitle = {Handbook of mathematical fluid dynamics. Vol. 1},
	isbn = {0-444-50330-7},
	pages = {71--305},
	year = {2002},
	publisher = {Amsterdam: Elsevier},
	language = {English},
	zbMATH = {1942873},
	Zbl = {1170.82369}
}

@article{C16,
	author = {Choi, Young-Pil},
	title = {Global classical solutions of the {Vlasov}-{Fokker}-{Planck} equation with local alignment forces},
	fjournal = {Nonlinearity},
	journal = {Nonlinearity},
	issn = {0951-7715},
	volume = {29},
	number = {7},
	pages = {1887--1916},
	year = {2016},
	language = {English},
	doi = {10.1088/0951-7715/29/7/1887},
	zbMATH = {6614383},
	Zbl = {1348.35275}
}

@article{CHY25,
	author = {Choi, Young-Pil and Hwang, Byung-Hoon and Yoo, Yeongseok},
	title = {Global existence of weak solutions to the nonlinear {Vlasov}-{Fokker}-{Planck} equation},
	fjournal = {Journal of Differential Equations},
	journal = {J. Differ. Equations},
	issn = {0022-0396},
	volume = {444},
	pages = {53},
	note = {Id/No 113573},
	year = {2025},
	language = {English},
	doi = {10.1016/j.jde.2025.113573},
	zbMATH = {8098491},
	Zbl = {1572.35260}
}

@article{MT11,
	author = {Motsch, Sebastien and Tadmor, Eitan},
	title = {A new model for self-organized dynamics and its flocking behavior},
	fjournal = {Journal of Statistical Physics},
	journal = {J. Stat. Phys.},
	issn = {0022-4715},
	volume = {144},
	number = {5},
	pages = {923--947},
	year = {2011},
	language = {English},
	doi = {10.1007/s10955-011-0285-9},
	zbMATH = {5976134},
	Zbl = {1230.82037}
}

@article{LY21,
	author = {Liao, Jie and Yang, Xiongfeng},
	title = {Stability of global {Maxwellian} for fully nonlinear {Fokker}-{Planck} equations},
	fjournal = {Journal of Statistical Physics},
	journal = {J. Stat. Phys.},
	issn = {0022-4715},
	volume = {185},
	number = {3},
	pages = {27},
	note = {Id/No 23},
	year = {2021},
	language = {English},
	doi = {10.1007/s10955-021-02844-9},
	zbMATH = {7445220},
	Zbl = {1483.35277}
}

@article{ALTPP00,
	author = {Andries, Pierre and Le Tallec, Patrick and Perlat, Jean-Philippe and Perthame, Beno{\^{\i}}t},
	title = {The {Gaussian}-{BGK} model of {Boltzmann} equation with small {Prandtl} number},
	fjournal = {European Journal of Mechanics. B. Fluids},
	journal = {Eur. J. Mech., B, Fluids},
	issn = {0997-7546},
	volume = {19},
	number = {6},
	pages = {813--830},
	year = {2000},
	language = {English},
	doi = {10.1016/S0997-7546(00)01103-1},
	zbMATH = {1574429},
	Zbl = {0967.76082}
}

@article{BGK54,
	author = {Bhatnagar, P. L. and Gross, E. P. and Krook, M.},
	title = {A model for collision processes in gases. {I}: {Small} amplitude processes in charged and neutral one-component systems},
	fjournal = {Physical Review, II. Series},
	journal = {Phys. Rev., II. Ser.},
	issn = {0031-899X},
	volume = {94},
	pages = {511--525},
	year = {1954},
	language = {English},
	doi = {10.1103/PhysRev.94.511},
	zbMATH = {3087985},
	Zbl = {0055.23609}
}

@article{P89,
	author = {Perthame, B.},
	title = {Global existence to the {BGK} model of {Boltzmann} equation},
	fjournal = {Journal of Differential Equations},
	journal = {J. Differ. Equations},
	issn = {0022-0396},
	volume = {82},
	number = {1},
	pages = {191--205},
	year = {1989},
	language = {English},
	doi = {10.1016/0022-0396(89)90173-3},
	zbMATH = {4136542},
	Zbl = {0694.35134}
}

@article{D94,
	author = {Dolbeault, J.},
	title = {Kinetic models and quantum effects: {A} modified {Boltzmann} equation for {Fermi}-{Dirac} particles},
	fjournal = {Archive for Rational Mechanics and Analysis},
	journal = {Arch. Ration. Mech. Anal.},
	issn = {0003-9527},
	volume = {127},
	number = {2},
	pages = {101--131},
	year = {1994},
	language = {English},
	doi = {10.1007/BF00377657},
	url = {zenodo.org/record/886461},
	zbMATH = {687666},
	Zbl = {0808.76084}
}
\end{document}